\documentclass{amsart}

\usepackage{amssymb}
\usepackage{mathrsfs}
\usepackage{stmaryrd}

\usepackage{hyperref}
\usepackage{enumitem}

\usepackage{booktabs,caption}
\usepackage{longtable}

\usepackage{comment}
\usepackage{graphicx}
\usepackage{color}
\usepackage[all]{xy}

\usepackage[alphabetic,nobysame]{amsrefs}

\newtheorem{theorem}{Theorem}[section]
\newtheorem{lemma}[theorem]{Lemma}
\newtheorem{proposition}[theorem]{Proposition}

\newtheorem{claim}[theorem]{Claim}

\newtheorem{maintheorem}{Theorem}

\theoremstyle{definition}
\newtheorem{definition}[theorem]{Definition}
\newtheorem{remark}[theorem]{Remark}
\newtheorem{example}[theorem]{Example}
\newtheorem{construction}[theorem]{Construction}

\numberwithin{equation}{section}

\renewcommand{\setminus}{\smallsetminus}

\newcommand{\epsi}{\varepsilon}
\newcommand{\phiv}{\varphi}
\newcommand{\into}{\hookrightarrow} 
\newcommand{\id}{\mathrm{id}} 
\newcommand{\GL}{\mathrm{GL}} 
\DeclareMathOperator{\embdim}{embdim}
\newcommand{\Gm}{\mathbb{G}_\mathrm{m}}
\newcommand{\al}{\alpha}

\DeclareMathOperator\Aut{Aut}

\DeclareMathOperator{\Hom}{Hom} 
\newcommand{\Ext}{\mathrm{Ext}}
\newcommand{\cExt}{\mathcal{E}xt}
\newcommand{\cHom}{\mathcal{H}om}
\DeclareMathOperator{\Spec}{Spec} 
\DeclareMathOperator{\Proj}{Proj} 

\def\pow#1{ \llbracket  #1 \rrbracket }

\newcommand{\stack}[1]{ \mathcal{M}^\mathrm{Kss}_{#1} } 
\newcommand{\modspace}[1]{M^\mathrm{Kps}_{#1}} 

\newcommand{\art}{(\mathrm{Art}_{\Bbbk})}
\newcommand{\comp}{(\mathrm{Comp}_{\Bbbk})}
\newcommand{\compC}{(\mathrm{Comp}_\mathbb{C})}
\newcommand{\set}{(\mathrm{Set})}

\DeclareMathOperator{\Spf}{Spf}
\newcommand{\Def}[1]{\mathrm{Def}_{#1}}

\newcommand{\grpd}{(\mathrm{Grpd})}
\newcommand{\gDef}[1]{\mathcal{D}ef_{#1}}

\newcommand{\DefqG}[1]{\mathrm{Def}^{\mathrm{qG}}_{#1}}
\newcommand{\TTqG}[2]{\mathbb{T}^{\mathrm{qG}, #1}_{#2}}
\newcommand{\cTqG}[2]{\mathscr{T}^{\mathrm{qG}, #1}_{#2}}

\def\conv#1{\mathrm{conv} \left\{ #1  \right\} } 

\newcommand\cM{\mathcal{M}}
\newcommand\cO{\mathcal{O}}

\newcommand\cU{\mathcal{U}}

\newcommand\cT{\mathscr{T}}

\newcommand\CC{\mathbb{C}}

\newcommand{\KK}{\Bbbk}
\newcommand{\LL}{\mathbb{L}}
\newcommand\NN{\mathbb{N}}
\newcommand\PP{\mathbb{P}}
\newcommand\QQ{\mathbb{Q}}
\newcommand\RR{\mathbb{R}}
\newcommand\TT{\mathbb{T}}
\newcommand\ZZ{\mathbb{Z}}

\newcommand\rH{\mathrm{H}}

\newcommand\rV{\mathrm{V}}

\newcommand{\frakm}{\mathfrak{m}}
\newcommand{\frakp}{\mathfrak{p}}

\newcommand{\frakX}{\mathfrak{X}}
\newcommand{\frakY}{\mathfrak{Y}}

\title[On deformation spaces of toric varieties]{On deformation spaces of toric singularities and on singularities of K-moduli of Fano varieties}

\author{Andrea Petracci}

\address{Dipartimento di Matematica, Universit\`a di Bologna, Piazza di Porta San Donato 5, 40126 Bologna, Italy}

\email{a.petracci@unibo.it}

\begin{document}

   \begin{abstract}
	Firstly, we see that the bases of the miniversal deformations of isolated $\QQ$-Gorenstein toric singularities are quite restricted.
	In particular, we classify the analytic germs of embedding dimension $\leq 2$ which are the bases of the miniversal deformations of isolated $\QQ$-Gorenstein toric singularities.
	
	Secondly, we show that the deformation spaces of isolated Gorenstein toric $3$-fold singularities appear, in a weak sense, as singularities of the K-moduli stack of K-semistable Fano varieties of every dimension $\geq 3$.
	As a consequence, we prove that the number of local branches of the K-moduli stack of K-semistable Fano varieties and of the K-moduli space of K-polystable Fano varieties is unbounded in each dimension $\geq 3$.
\end{abstract}

\maketitle

\section{Introduction}
\label{sec:introduction}

   The setting of this paper is twofold: isolated $\QQ$-Gorenstein toric singularities and K-polystable toric Fano varieties.
   The symbol $\CC$ denotes an algebraically closed field of characteristic $0$.

   \subsection{Deformations of isolated $\QQ$-Gorenstein toric singularities} \label{sec:intro_toric_singularities}

   Studying the bases of miniversal  deformations (aka Kuranishi families) is  important to understand the local properties of moduli spaces.
   Hypersurface singularities, or more generally local complete intersection singularities, have unobstructed deformations, i.e.\ the bases of their miniversal deformations are smooth.
   Isolated quotient singularities of dimension $\geq 3$ are rigid \cite{schlessinger}, so the base of their miniversal deformations is just a point.
   
   On the other hand, more complicated singularities may have almost arbitrarily complicated miniversal deformations: Vakil \cite{vakil} has shown that the formal spectrum of every noetherian complete local  ring which comes via completion and base change from a scheme of finite type over $\ZZ$ is, up to a smooth factor, the base of the miniversal deformation of some isolated normal Cohen--Macaulay $3$-fold singularity. This is the so called \emph{Murphy's law} for versal deformation spaces of isolated normal Cohen--Macaulay $3$-fold singularities.
   
   Toric affine varieties provide normal Cohen--Macaulay singularities which can be much more complicated than local complete intersections or quotient singularities. Nevertheless their combinatorial nature makes the problem of explicitly computing miniversal deformations tractable.
   Thanks to the seminal work by Altmann \cite{altmann_versal} we give a characterisation of noetherian complete local rings with embedding dimension $\leq 2$ which are the hulls of the infinitesimal deformation functor of isolated $\QQ$-Gorenstein toric singularities:

   \begin{maintheorem} \label{thm:main_toric_singularity}
   	Let $R$ be a noetherian complete local  $\CC$-algebra with residue field $\CC$ and with embedding dimension $\leq 2$.
Then the formal spectrum of $R$ is the base of the miniversal deformation of an isolated $\QQ$-Gorenstein toric singularity over $\CC$ if and only if $R$ is isomorphic to (exactly) one of the following $\CC$-algebras:
\begin{enumerate}
	\item[0)] $\CC $,
	\item[1.a)] $\CC [x] / (x^2)$,
	\item[1.b)] $\CC \pow{x}$,
	\item[2.a)] $\CC [x,y]/(x^2, y^2)$,
	\item[2.b)] $\CC [x,y] / (x^2, xy, y^3)$,
	\item[2.c)] $\CC \pow{x,y} / (x^2, xy)$,
	\item[2.d)] $\CC \pow{x,y}$.
\end{enumerate}
   \end{maintheorem}

Moreover, for arbitrary embedding dimension (not necessarily $\leq 2$) we  find strict necessary conditions that a ring must satisfy in order to be the hull of the deformation functor of an isolated Gorenstein toric $3$-fold singularity (see Proposition~\ref{prop:consequences_thm_altmann}). A consequence is the following:

\begin{maintheorem} \label{thm:main_no_murphys_law_for_toric_singularities}
	For every $n \geq 1$, the ring $\CC \pow{x_1, \dots, x_n} / (x_1^3)$ is not the hull of any isolated $\QQ$-Gorenstein toric singularity over $\CC$.
	Therefore Murphy's law in the sense of Vakil~\cite{vakil} does not hold for isolated $\QQ$-Gorenstein toric singularities.	
\end{maintheorem}

The proofs of the two theorems above use Altmann's results \cite{altmann_versal} on the versal deformation of isolated Gorenstein toric $3$-fold singularities (see \S\ref{sec:altmann}).
Altmann's results are partially extended to more general toric singularities in \cite{altmann_kastner, klaus_alex_matej_1, klaus_alex_matej_2, matej, filip_gerstenhaber}.
A conjectural description of the smoothing components of a (not necessarily isolated) Gorenstein toric $3$-fold singularity is given in \cite{cofipe}.

   \subsection{Deformations of K-polystable toric Fano varieties} \label{sec:intro_Fano}
   
   A \emph{Fano variety} is a normal projective variety $X$ over $\CC$ such that its anticanonical divisor $-K_X$ is $\QQ$-Cartier and ample.
   A \emph{del Pezzo surface} is a Fano variety of dimension $2$.
   All Fano varieties considered in this article have log terminal singularities.
   When we write \emph{Fano $n$-fold} we mean a Fano variety of dimension $n$; in particular, a Fano $n$-fold can be singular.
   If $X$ is a Fano $n$-fold, then its \emph{anticanonical volume} is the intersection number $(-K_X)^n$ and is a positive rational number.

   Recently there has been a lot of significant progress on constructing moduli spaces of Fano varieties using K-stability \cite{ABHLX, xu_minimizing, BLX, jiang_boundedness, blum_xu_uniqueness, xu_zhuang, properness_K_moduli, projectivity_K_moduli_final, lwx}: it is known that, for each integer $n \geq 1$ and for every rational number $v>0$, there exists an Artin stack $\cM^\mathrm{Kss}_{n,v} $ which is  of finite type over $\CC$ and which parametrises K-semistable Fano $n$-folds with anticanonical volume $v$; moreover this stack admits a good moduli space $M^\mathrm{Kps}_{n,v} $, which is a projective scheme over $\CC$. The closed points of $M^\mathrm{Kps}_{n,v} $ are in one-to-one correspondence with the K-polystable Fano $n$-folds with anticanonical volume $v$.
   We refer the reader to \cite{xu_survey} for a survey on these topics.

For every integer $n \geq 1$, we consider the following countable disjoint unions:
\begin{equation*}
	\stack{n} := \coprod_{v \in \QQ_{>0}} \cM^\mathrm{Kss}_{n,v} \qquad \text{and} \qquad 
	\modspace{n} := \coprod_{v \in \QQ_{>0}} M^\mathrm{Kps}_{n,v}.
\end{equation*}
In this way we get an Artin stack (resp.\ a scheme) which is locally of finite type over $\CC$ and which parametrises all K-semistable (resp.\ K-polystable) Fano $n$-folds, because the anticanonical volume is locally constant on a $\QQ$-Gorenstein family of Fano varieties. The stack $\stack{n}$ is called K-moduli stack and the scheme $\modspace{n}$ is called K-moduli space.
It is interesting to study the geometry (of the connected components) of $\stack{n}$ and of $\modspace{n}$; in this article we focus on some local properties.

If $n=2$, then the K-moduli stack $\stack{2}$, which parametrises K-semistable del Pezzo surfaces, is smooth because $\QQ$-Gorenstein deformations of del Pezzo surfaces are unobstructed \cite{hacking_prokhorov, procams}.  Thus the K-moduli space $\modspace{2}$ is a disjoint union of normal projective varieties with rational singularities (see \cite[Proposition~2.3]{ask_petracci}).

An analogous result holds when considering \emph{smooth} Fano varieties, because deformations of smooth Fano varieties are unobstructed.
Therefore $\stack{n}$ (resp.\ $\modspace{n}$) is smooth (resp.\ normal) in a neighbourhood of a smooth K-semistable (resp.\ K-polystable) Fano $n$-fold.
The same holds for terminal Fano $3$-folds, because deformations of terminal Fano $3$-folds are unobstructed by \cite{sano}.

Roughly speaking, we can say that deformations of Fano varieties which are mildly singular or have low dimension are unobstructed. Therefore in a neighbourhood of such Fano varieties the K-moduli stack (resp.\ K-moduli space) is smooth (resp.\ normal). Moreover, by \cite{reductive_quotients_klt} the K-moduli space is of klt type in a neighbourhood of every K-polystable Fano variety whose $\QQ$-Gorenstein deformations are unobstructed.

         In joint work with Kaloghiros \cite{ask_petracci} we exhibited the first examples of singular points on $\stack{n}$, for each $n \geq 3$.
  More precisely, we constructed K-polystable toric Fano varieties with obstructed $\QQ$-Gorenstein deformations. In this way we showed that, if $n \geq 3$, $\stack{n}$ and $\modspace{n}$ can be locally reducible or non-reduced (or both).
  
  Even if one considers only the irreducible components of $\modspace{n}$, for $n \geq 3$, which generically parametrises smooth K-polystable Fano $n$-folds, one can get non-normal singularities on the K-moduli space at the points which correspond to singular Fano varieties. Indeed, for each $n \geq 3$, the example in \cite[Theorem~1.1]{ask_petracci} is a singular K-polystable toric Fano $n$-fold with Gorenstein canonical singularities which lies at the intersection of  three distinct irreducible components of $\modspace{n}$, each of which generically parametrises smooth Fano $n$-folds.

      Here we prove a much more general statement concerning versal $\QQ$-Gorenstein deformations of K-polystable toric Fano varieties, and consequently concerning the singularities that appear on K-moduli of Fano varieties:
      
         \begin{maintheorem} \label{thm:main_hull_K-moduli_stack}
      	Let $R$ be a noetherian complete local $\CC$-algebra with residue field $\CC$.
      	Assume that $R$ is the hull of the deformation functor of an isolated Gorenstein toric $3$-fold singularity over $\CC$.
      	
      	Then, for every integer $n \geq 3$, there exists a K-polystable toric Fano $n$-fold $X$ over $\CC$ such that
      	the hull of the $\QQ$-Gorenstein deformation functor of $X$ is isomorphic to $(R \widehat{\otimes}_\CC R) \pow{t_1, \dots, t_m} / I$, for some non-negative integer $m$ and for some ideal $I$ which contains no non-zero constants.
      \end{maintheorem}
  
  	Some remarks are in order.
  	\begin{itemize}
  		\item Let $R$ be a noetherian complete local $\CC$-algebra with residue field $\CC$ and with maximal ideal $\frakm$.
  		The symbol $R \widehat{\otimes}_\CC R$ used in Theorem~\ref{thm:main_hull_K-moduli_stack} denotes the completed tensor product of $R$ with itself; this is the inverse limit of $R/\frakm^{n+1} \otimes_\CC R/\frakm^{n+1}$ as $n\geq 0$, where $\otimes_\CC$ denotes the ordinary tensor product.
  		In this way, if $R$ is a power series ring over $\CC$ in $r$ variables, then $R \widehat{\otimes}_\CC R$ is a power series ring over $\CC$ in $2r$ variables.

  		\item Let $R$ and $R'$ be two noetherian complete local $\CC$-algebras with residue field $\CC$. Then $R'$ is isomorphic to $(R \widehat{\otimes} R) \pow{t_1, \dots, t_m} / I$, for some non-negative integer $m$ and for some ideal $I$ with $I \cap (R \widehat{\otimes} R) =  0 $, if and only if there exist two local $\CC$-algebra homomorphisms $f \colon R \widehat{\otimes}_\CC R \to R'$
  		and $g \colon R' \to R \widehat{\otimes}_\CC R$ such that $g \circ f = \id_{R \widehat{\otimes}_\CC R}$, i.e.\ if there exists a surjection
  		\[
  		\Spf R' \longrightarrow \Spf R \widehat{\otimes}_\CC R = \Spf R \times_{\Spec \CC} \Spf R
  		\]
  		admitting a section.
  		If this is the case, then $R \widehat{\otimes}_\CC R$ is a subring of $R'$.
  		
  		\item Thanks to what we have just seen, Theorem~\ref{thm:main_hull_K-moduli_stack} can be rephrased as follows.       If a complex analytic germ $(0 \in S)$ is the base of the miniversal deformation of an isolated Gorenstein toric $3$-fold singularity, then for each $n \geq 3$ it is possible to find a K-polystable toric Fano $n$-fold $X$ such that its miniversal $\QQ$-Gorenstein deformation space $\DefqG{X}$ satisfies the following: there exists a surjection of analytic germs
  		\[
  		\DefqG{X} \longrightarrow S \times S
  		\]
  		which admits a section. (A surjection of germs which admits a section is called \emph{retraction} by Jelisiejew~\cite{joachim}.)
  		
  		\item Since the germ $\DefqG{X}$ is the local structure of the stack $\stack{n}$ near the closed point to $[X]$, Theorem~\ref{thm:main_hull_K-moduli_stack} implies that the germ $S \times S$ appears, up to retraction, as a singularity on $\stack{n}$, for every $n \geq 3$.
  		This is what we meant in the abstract when we wrote that deformation spaces of isolated Gorenstein toric $3$-fold singularities appear \emph{in a weak sense} as singularities of the K-moduli stack of K-semistable Fano varieties of dimension $\geq 3$.
  		
  		\item The proof of Theorem~\ref{thm:main_hull_K-moduli_stack} will show that the ring $(R \widehat{\otimes}_\CC R) \pow{t_1, \dots, t_m} / I$ can be chosen independently from the dimension $n$.
  	\end{itemize}


  Examples of hulls of isolated Gorenstein toric $3$-fold singularities are given in \S\ref{sec:toric_singularities}.
  Since such hulls are quite restricted, we might expect that toric geometry does not help to prove that K-moduli of Fano varieties satisfy Murphy's law in the sense of Vakil~\cite{vakil} or Murphy's law up to retraction in the sense of Jelisiejew~\cite{joachim}.
Actually, it is not clear to us whether Murphy's law should hold for K-moduli of Fano varieties at all.

Nonetheless, hulls of toric singularities are enough to show that K-moduli of Fano varieties can acquire an arbitrarily large number of local branches:

\begin{maintheorem} \label{thm:main_many_branches}
For every integer $n \geq 3$ and for every integer $m \geq 1$, there exists a K-polystable toric Fano $n$-fold $X$ over $\CC$ such that $\stack{n}$ and $\modspace{n}$ have at least $m$ local branches at the point $[X]$.
\end{maintheorem}

	In this paper we do not consider deformations of toric varieties which are non-affine and non-Fano.
	Concerning deformations of smooth projective toric varieties, the following results are known:
	smooth projective toric surfaces are unobstructed \cite{ilten_surfaces},
	there exist smooth projective toric $3$-folds which are obstructed \cite{ilten_turo} (see also \cite{petracci_felten_robins}),
	the deformation space of a smooth projective toric variety cannot be a non-reduced point \cite{ilten_vollmert}.

\subsection*{Outline}
The setting of isolated $\QQ$-Gorenstein toric singularities (as in \S\ref{sec:intro_toric_singularities}) is studied in \S\ref{sec:toric_singularities}.
The proof of Theorem~\ref{thm:main_toric_singularity} and
the proof of Theorem~\ref{thm:main_no_murphys_law_for_toric_singularities} is given in \S\ref{sec:subsection_toric_singularities_final}.

The setting of K-polystable toric Fano varieties (as in \S\ref{sec:intro_Fano}) is studied in \S\ref{sec:toric_Fano}.
The proof of Theorem~\ref{thm:main_hull_K-moduli_stack} is given in \S\ref{sec:subsection_Fano}.
The proof of Theorem~\ref{thm:main_many_branches} is given in \S\ref{sec:arbitrarily_many_branches}.

In \S\ref{sec:deformations} we give some preliminaries on deformation theory.
In \S\ref{sec:classical_deformation} we present some basics on deformation theory and we fix the terminology.
In \S\ref{sec:restriction_maps} and \S\ref{sec:qG} we study restriction maps between deformation functors and we consider $\QQ$-Gorenstein deformations, respectively -- both sections will be used only with Fano varieties, i.e.\ in \S\ref{sec:toric_Fano} and not in the proof of Theorems~\ref{thm:main_toric_singularity} and \ref{thm:main_no_murphys_law_for_toric_singularities}.

In \S\ref{sec:algebra} we give some tools in commutative algebra.
In \S\ref{sec:subsection_standard_graded} we recall what a standard graded algebra is and we state same properties relating a standard graded algebra to its completion at its homogeneous maximal ideal.
In \S\ref{sec:segre_products} we consider Segre products of standard graded algebras: this notion and its properties will be used only in the proof of Theorem~\ref{thm:main_many_branches} in \S\ref{sec:arbitrarily_many_branches}.
In \S\ref{sec:newton} we study generalised Newton identities which are crucial in order to study the standard graded algebras associated to polygons \`a la Altmann in \S\ref{sec:toric_singularities}.

\subsection*{Notation and conventions}
The set of non-negative integers is denoted by $\NN$.
The symbol $\CC$ stands for an algebraically closed field of characteristic $0$, whereas $\Bbbk$ denotes an arbitrary field.
Every ring is commutative with identity.
Every toric variety or toric singularity is assumed to be normal.
A polygon is by definition a polytope of dimension $2$.

\subsection*{Acknowledgements}
I wish to thank Klaus Altmann, Alexandru Constantinescu, Matej Filip, Anne-Sophie Kaloghiros, and Alessandro Oneto for fruitful conversations.
I am very grateful to Jan Christophersen, Marco Manetti, Jan Stevens, Duco van Straten, Nikolaos Tziolas for useful comments and observations about deformation theory.
Finally I thank Matteo Tanzi for help with the proof of Lemma~\ref{lem:infinite_iterates}.


\section{Deformations}
\label{sec:deformations}
\subsection{Quick recap of deformation theory}
\label{sec:classical_deformation}

Here we briefly recall some rudiments on deformation theory, just to fix the terminology.
Details and proofs can be found in any reference about deformation theory, e.g.\ \cite{schlessinger_artin_rings, sernesi, manetti_seattle, talpo_vistoli}.

Let $\set$ be the category of sets.
Let $\KK$ be a field and let $\art$ be the category of artinian local $\KK$-algebras with residue field $\KK$.
A \emph{deformation functor} is a covariant functor $F \colon \art \to \set$ such that $F(\KK)$ is a singleton  and $F$ satisfies Schlessinger's axioms (H1) and (H2).
One can define the notions of tangent space and of obstruction space of a deformation functor. 
One can define the notion of a map between deformation functors  and define when such a map is smooth.

Let $\comp$ denote the category of noetherian complete local $\KK$-algebras with residue field $\KK$.
Every ring $R$ in $\comp$ induces a deformation functor $h_R = \Hom_{\comp}(\cdot, R)$, which is called the functor prorepresented by $R$.
A ring $R$ in $\comp$ is called a \emph{hull} of a deformation functor $F$ if there exists a smooth map $h_R \to F$ which induces a bijection on the tangent spaces. A hull is unique, if it exists.

\medskip

For a scheme $X$ of finite type over $\KK$, let $\Def{X} \colon \art \to \set$ be the functor which maps each $A \in \art$ to the set of isomorphism classes of flat deformations of $X$ over $\Spec A$; it is a deformation functor.

If $R$ is a hull of $\Def{X}$, we say that the formal spectrum of $R$, denoted by $\Spf R$, is the \emph{base of the miniversal deformation} of $X$. If $\KK = \CC$ one could work in the analytic category and $R$ would be the completion of an analytic germ, which is called the base of the Kuranishi--Grauert deformation (or simply the Kuranishi space) of the analytification of $X$.
Both in the algebraic/formal category and in the analytic category we  say that $\Spf R$ or the corresponding analytic germ is the  \emph{(versal) deformation space} of $X$.
We say that $X$ is \emph{rigid} if its deformation space is a point, i.e.\ if the hull of its deformation functor is $\KK$.

If $\LL_X$ denotes the cotangent complex of $X$ over $\KK$, then for every $i \geq 0$ we consider the $\KK$-vector space $\TT^i_X = \Ext^i(\LL_X, \cO_X)$ and the coherent $\cO_X$-module $\cT^i_X = \cExt^i(\LL_X, \cO_X)$. The tangent space of $\Def{X}$ is $\Ext^1_X(\LL_X, \cO_X)$ and $\Ext^2_X(\LL_X, \cO_X)$ is an obstruction space of $\Def{X}$. Under certain hypotheses one can replace $\LL_X$ with $\tau_{\geq 0} \LL_X = \Omega_X$, which is the $\cO_X$-module of K\"ahler differentials of $X$ over $\KK$; for example:
\begin{itemize}
\item $\TT^0_X = \Hom(\Omega_X, \cO_X)$ and $\cT^0_X = \cHom (\Omega_X, \cO_X)$;
\item if $X$ is reduced and generically smooth over $\KK$ (e.g.\ if $\KK$ is perfect and $X$ is reduced), then $\TT^1_X = \Ext^1(\Omega_X, \cO_X)$ and $\cT^1_X = \cExt^1(\Omega_X, \cO_X)$;
\item if $\KK$ is perfect and ($X$ is normal or ($X$ is reduced and lci over $\KK$)), then $\TT^2_X = \Ext^2(\Omega_X, \cO_X)$ and $\cT^2_X = \cExt^2(\Omega_X, \cO_X)$.
\end{itemize}

\subsection{Restriction maps}
\label{sec:restriction_maps}
This section can be omitted if one is interested in the proof of Theorems~\ref{thm:main_toric_singularity} and \ref{thm:main_no_murphys_law_for_toric_singularities} only; it will be applied only in \S\ref{sec:qG}.


When $U$ is an open subscheme of a scheme $X$, then one can restrict a deformation of $X$ to $U$ and get a deformation of $U$. This produces a natural transformation of functors $\Def{X} \to \Def{U}$.

\begin{lemma} \label{lem:injective_product_restrictions}
Let $X$ be a separated scheme of finite type over $\KK$ and let $\cU = \{ U_i \}_{i \in I}$ be a finite affine open cover of $X$.
If $\rH^1(X, \cT^0_X) = 0$, then the product of restriction maps
\[
\Def{X} \longrightarrow \prod_{i \in I} \Def{U_i}
\]
is injective.
\end{lemma}

\begin{proof}
We need to prove that, for every $A \in \art$, if  $\xi$ and $\eta$ are two deformations of $X$ over $A$ such that $\xi \vert_{U_i}$ is isomorphic to $\eta \vert_{U_i}$ for every $i \in I$, then
$\xi$ and $\eta$ are isomorphic.
We prove this by induction on $\dim_\KK \frakm_A$, where $\frakm_A$ is the maximal ideal of $A$.
The case $\dim_\KK \frakm_A = 0$ (i.e.\ $A = \KK$) is obvious.

For the inductive step we consider a small extension $f \colon A' \to A$ in $\art$, with kernel $J$, and we assume that the thesis holds for $A$.
Let $\xi'$ and $\eta'$ be two deformations of $X$ over $A'$ which are locally isomorphic with respect to $\cU$, i.e.\ $\xi' \vert_{U_i} \simeq \eta' \vert_{U_i}$ for every $i \in I$.
We want to show that $\xi'$ and $\eta'$ are isomorphic.
We consider the pushforward $\xi = f_* (\xi')$ and $\eta = f_* (\eta')$ in $\Def{X}(A)$.
It is clear that $\xi$ and $\eta$ are locally isomorphic with respect to $\cU$.
Therefore, by inductive hypothesis, we have that $\xi$ and $\eta$ are isomorphic. From now on we identify $\eta$ with $\xi$.

Since the set of liftings of $\xi$ to $A'$ is a pseudo-torsor under the action of the abelian group $J \otimes_\KK \TT^1_X$, the two deformations $\xi'$ and $\eta'$  differ by the action of a unique element of $J \otimes_\KK \TT^1_X$, i.e.\ there exists a unique element $g \in J \otimes_\KK \TT^1_X$ such that $g + \xi' = \eta'$.
Since the set of liftings of $\xi \vert_{U_i}$ to $A'$ is a pseudo-torsor under $J \otimes_\KK \TT^1_{U_i}$ and $\xi' \vert_{U_i}$ and $\eta' \vert_{U_i}$ are isomorphic, we have $g \vert_{U_i} = 0$ in $J \otimes_\KK \TT^1_{U_i}$, for every $i \in I$.
This implies that $g$ lies in the kernel of the natural homomorphism $J \otimes_\KK \TT^1_X \to J \otimes_\KK \rH^0(X, \cT^1_X)$.
Now consider the short exact sequence
\[
0 \to \rH^1(X, \cT^0_X) \to \TT^1_X \to \rH^0(X, \cT^1_X) \to \rH^2(X, \cT^0_X)
\]
which comes from the local-to-global spectral sequence whose second page is $E_2^{p,q} = \rH^p(X, \cT^q_X)$ and which converges to $\TT^{p+q}_X$.
By tensoring the short exact sequence above with $J$ and by using the hypothesis that
 $\rH^1(X, \cT^0_X)$ vanishes, we deduce that $g = 0$ in $J \otimes_\KK \TT^1_X$.
Hence $\xi'$ and $\eta'$ are isomorphic.
This concludes the proof of the inductive step.
\end{proof}

\begin{proposition} \label{prop:section_deformation}
	Let $X$ be a separated scheme of finite type over $\KK$, let $x$ be an isolated point of the non-smooth locus of $X$ and let $Y$ be an affine open neighbourhood of $x$ in $X$ such that $Y \setminus \{ x \}$ is smooth.
	If $\rH^1(X, \cT^0_X) = 0$ and  $\rH^2(X, \cT^0_X) = 0$, then the restriction map
	\begin{equation} \label{eq:restriction_map}
	\Def{X} \longrightarrow \Def{Y}
	\end{equation}
	is surjective and admits a section.
\end{proposition}

\begin{remark} \label{rmk:restriction_map_smooth?}
	We do not claim that the restriction map~\eqref{eq:restriction_map} is smooth.
	Using the local-to-global spectral sequence of Ext, it is easy to show that, under the hypotheses of Proposition~\ref{prop:section_deformation}, if $X \setminus \{x\}$ is lci and $\rH^1(X, \cT^1_X) = 0$, then the restriction map~\eqref{eq:restriction_map} is smooth.
\end{remark}

\begin{proof}[Proof of Proposition~\ref{prop:section_deformation}]
	Let $\cU = \{ U_i \}_{i \in I}$ be an affine open cover of $X \setminus \{ x \}$,
	where the index set $I$ is finite and equipped with a total order.
	Hence $\{ Y \} \cup \cU$ is a finite affine open cover of $X$.
	We consider the product of the restriction maps
	\begin{equation} \label{eq:product_restriction_map}
	\Def{X} \longrightarrow \Def{Y} \times \prod_{i \in I} \Def{U_i}
	\end{equation}
which is injective by Lemma~\ref{lem:injective_product_restrictions}.

\begin{claim} \label{claim:image}
Let $A \in \art$, let $\xi \in \Def{Y}(A)$ be an arbitrary deformation of $Y$ over $A$ and, for each $i \in I$, let $\xi_i \in \Def{U_i}(A)$ denote the trivial deformation of $U_i$ over $A$.
Then the tuple $(\xi, (\xi_i)_{i \in I}   )$ lies in the image of the map~\eqref{eq:product_restriction_map}.
\end{claim}

In order to prove this claim, we need to construct a deformation $\eta \in \Def{X}(A)$ such that the restriction $\eta \vert_Y$ is isomorphic to $\xi$ and the restriction $\eta \vert_{U_i}$ is trivial for each $i \in I$.
We proceed by induction on $\dim_\KK \frakm_A$, where $\frakm_A$ is the maximal ideal of $A$.
The case $\dim_\KK \frakm_A = 0$ is obvious.

For the inductive step we consider a small extension $f \colon A' \to A$ in $\art$, with kernel $J$, and we assume that the claim holds for $A$.
We have an arbitrary deformation $\xi' \in \Def{Y}(A')$ and trivial deformations $\xi_i' \in \Def{U_i}(A')$ and we want to show that there exists a deformation $\eta' \in \Def{X}(A')$ such that $\eta' \vert_Y \simeq \xi'$ and $\eta' \vert_{U_i}$ is trivial for each $i$.
By inductive hypothesis there exists a deformation $\eta \in \Def{X}(A)$ such that $\eta \vert_Y \simeq f_*(\xi')$ and $\eta \vert_{U_i}$ is trivial for each $i$.

	We consider the datum of the trivial deformations of $U_i$ over $A'$ and of the deformation $\xi'$; these are deformations, over $A'$, of the elements of the affine open cover $\cU \cup \{Y\}$ of $X$ and they are pairwise compatible, i.e.\ their restrictions to the double intersections $U_i \cap U_j$ and $U_i \cap Y$ are pairwise isomorphic.
This is true because the $U_i \cap Y$ are smooth and affine, therefore each deformation of $U_i \cap Y$ is trivial.
By \cite[Proposition~3.12]{osserman} this datum induces an obstruction class in $J \otimes_\KK \rH^2(X, \cT^0_X)$.
The vanishing of $\rH^2(X, \cT^0_X)$ implies that our datum is obtained from a deformation $\eta'$ of $X$ over $A'$.
This concludes the proof of Claim~\ref{claim:image}.

Combining the injectivity of \eqref{eq:product_restriction_map} and Claim~\ref{claim:image} we have that the assignment
\[
\xi \mapsto \left( \xi, (\text{trivial deformation of } U_i)_{i \in I} \right)
\]
for all deformations $\xi$ of $Y$ produces a section of the restriction map~\eqref{eq:restriction_map}.
	This concludes the proof of Proposition~\ref{prop:section_deformation}.
\end{proof}

\begin{remark} \label{rmk:section_proposition_finitely_many}
	It is clear that we can generalise Proposition~\ref{prop:section_deformation} to  finitely many isolated points of the non-smooth locus of $X$.
\end{remark}


\subsection{$\QQ$-Gorenstein deformations} \label{sec:qG}
This section can be omitted if one is interested in the proof of Theorems~\ref{thm:main_toric_singularity} and \ref{thm:main_no_murphys_law_for_toric_singularities} only; it will be applied only in \S\ref{sec:toric_Fano}.
	
Fix a field $\KK$ of characteristic $0$ and fix a $\QQ$-Gorenstein normal variety $X$ over $\KK$.
The \emph{canonical cover stack} of $X$ is a certain Deligne--Mumford stack $\frakX$ which has $X$ as a coarse moduli space and which is   defined in \cite{abramovich_hassett_stable_varieties}.
Its fundamental properties are: the construction of $\frakX$ is local on $X$,
the stack $\frakX$ is normal and Gorenstein,
and the structure morphism $\epsi \colon \frakX \to X$ is an isomorphism on the Gorenstein locus of $X$.
Notice that $\epsi$ is proper and cohomologically affine.

By definition \cite{kollar_shepherd_barron}, $\QQ$-Gorenstein deformations of $X$ are deformations of the canonical cover stack $\frakX$ of $X$:
\[
\DefqG{X} = \Def{\frakX}.
\]
Therefore $\QQ$-Gorenstein deformations of $X$ are controlled by the $\KK$-vector spaces $\TTqG{i}{X} = \Ext^i(\Omega_\frakX, \cO_\frakX)$ and the coherent $\cO_X$-modules $\cTqG{i}{X} = \epsi_\star \cExt^i(\Omega_\frakX, \cO_\frakX)$, for $i=0,1,2$.
It is always true that $\TTqG{0}{X} = \TT^0_X$ and $\cTqG{0}{X} = \cT^0_X$.

If $X$ is Gorenstein, then the canonical cover stack is isomorphic to $X$, hence every deformation of $X$ is $\QQ$-Gorenstein, i.e.\ $\DefqG{X} = \Def{X}$.

Now, in the context of $\QQ$-Gorenstein deformations, we consider restriction maps similarly to \S\ref{sec:restriction_maps}.
	
\begin{proposition} \label{prop:section_deformation_qG}
Let $\KK$ be a field of characteristic $0$ and let $X$ be a $\QQ$-Gorenstein normal variety over $\KK$.
Let $x_1, \dots, x_t$ be finitely many isolated points of the non-smooth locus of $X$. For each $i=1,\dots, t$, let $Y_i$ be an affine open neighbourhood of $x_i$ in $X$ such that $Y_i \setminus \{ x_i \}$ is smooth.
If $\rH^1(X, \cT^0_X) = 0$ and  $\rH^2(X, \cT^0_X) = 0$, then the product of the restriction maps
\begin{equation*} 
	\DefqG{X} \longrightarrow \prod_{i=1}^t \DefqG{Y_i}
\end{equation*}
is surjective and admits a section.
\end{proposition}

We remark that the points $x_1, \dots, x_t$ need not be all the singular points of $X$. In other words, $X \setminus \{ x_1, \dots, x_t \}$ can be singular and can have non-isolated singularities.

\begin{proof}[Proof of Proposition~\ref{prop:section_deformation_qG}]
Consider the canonical cover stack $\epsi \colon \frakX \to X$.
We have that $\frakX$ is a separated Deligne--Mumford stack of finite type over $\KK$.
Since $\epsi$ is cohomologically affine and $\epsi_\star \cT^0_\frakX = \cT^0_X$, the hypothesis of the vanishing of $\rH^1(X, \cT^0_X)$ and of  $\rH^2(X, \cT^0_X)$
implies the vanishing of $\rH^1(\frakX, \cT^0_\frakX)$ and of $\rH^2(\frakX, \cT^0_\frakX)$.

For each $i$, take the preimage $\frakY_i = \epsi^{-1} (Y_i) \subseteq \frakX$.
Obviously $\frakY_i \setminus \{ x_i \} \simeq Y_i \setminus \{ x_i \}$ is smooth.
If we were able to apply Proposition~\ref{prop:section_deformation} (and Remark~\ref{rmk:section_proposition_finitely_many}) verbatim to $\frakX$, we would get
that the product of the restriction maps
\begin{equation*} 
	\Def{\frakX} \longrightarrow \prod_{i=1}^t \Def{\frakY_i}
\end{equation*}
is surjective and admits a section -- this would be a reformulation of the thesis.

Therefore we need to convince ourselves that, with appropriate (but small) modifications, we can apply Lemma~\ref{lem:injective_product_restrictions} and Proposition~\ref{prop:section_deformation} to separated Deligne--Mumford stacks. So now we consider a separated Deligne--Mumford stack $\frakX$ of finite type over a characteristic zero field $\KK$ and we consider its coarse moduli space $\epsi \colon \frakX \to X$.

If we take an affine open cover $\cU_X = \{ U_i \}_{i}$ of $X$, by taking the preimages we obtain a Zariski open cover $\cU_\frakX = \{ \epsi^{-1}(U_i) \}_{i}$ of $\frakX$. If $\frakX$ is not a scheme, there exists an element in the open cover $\cU_\frakX$ which is not an affine scheme.
However, each element of $\cU_\frakX$ is a separated Deligne--Mumford stack whose coarse moduli space is an affine scheme, namely one of the $U_i$'s. Since $\epsi^{-1}(U_i) \to U_i$ is cohomologically affine, we have that every quasi-coherent sheaf on $\epsi^{-1}(U_i)$  does not have cohomology in positive degree.
Therefore, in order to compute the cohomology of quasi-coherent sheaves of $\frakX$, we can use the \v{C}ech cohomology with respect to $\cU_\frakX$.
This is the only observation needed to make Lemma~\ref{lem:injective_product_restrictions} and Proposition~\ref{prop:section_deformation} work for separated Deligne--Mumford stacks of finite type over a field of characteristic $0$.
\end{proof}

\section{Algebraic intermezzo}
\label{sec:algebra}
\subsection{Standard graded algebras} \label{sec:subsection_standard_graded}

Fix an arbitrary field $\KK$.

\begin{definition}
A \emph{standard graded $\KK$-algebra} is an $\NN$-graded $\KK$-algebra $A = \bigoplus_{n \geq 0} A_n$  of finite type over $\KK$, generated in degree $1$, and with $A_0 = \KK$.
\end{definition}

\begin{lemma} \label{lem:minimal_primes_completion}
	Let $A$ be a standard graded $\KK$-algebra.
	Let $(R, \frakm)$ be the completion of $A$ at the maximal ideal of $A$ generated by the homogeneous elements with positive degree. Then the following statements hold.
	\begin{enumerate}
		\item The associated graded ring $\mathrm{gr}_\frakm R$ is isomorphic to $A$ as an $\NN$-graded $\KK$-algebra.
		\item $A$ is a domain if and only if $R$ is a domain.
		\item If $P$ is a homogeneous prime ideal of $A$, then the extension $PR$ is a prime ideal of $R$ and $PR \cap A = P$.
		\item Extension and contraction of ideals give a $1$-to-$1$ correspondence between minimal primes of $A$ and minimal primes of $R$.
	\end{enumerate}
\end{lemma}

\begin{proof}
	Let $M \subset A$ denote the maximal ideal of $A$ generated by the homogeneous elements with positive degree.
	Let $A_M$ be the localisation of $A$ at $M$ and let $M A_M$ be the maximal ideal of $A_M$.
	Since we are working over the field $\KK \simeq A/M = A_0$ it is easy to show that the localisation homomorphism $A \to A_M$ is injective. 
	The completion homomorphism $A_M \to R$ is faithfully flat and injective.
	Therefore $A \to R$ is flat and injective.
	
	(1) For every $n \geq 0$, we have the following chain of isomorphisms:
\begin{equation*}
\frakm^n / \frakm^{n+1} \simeq (MA_M)^n / (M A_M)^{n+1} 
\simeq M^n / M^{n+1} \otimes_A A_M 
\simeq M^n / M^{n+1} 
\simeq A_n.
\end{equation*}
	By taking the direct sum for $n\geq 0$ and observing that these isomorphisms are compatible with the multiplicative structure we conclude. See \cite[Theorem~13.8(iii)]{matsumura}.
	
	(2) Since $A$ is a subring of $R$, if $R$ is a domain, then $A$ is a domain.
	In order to prove the opposite implication, let us assume that $A$ is a domain.
	By (1) we have that $(R, \frakm)$ is a local noetherian ring such that $\mathrm{gr}_\frakm R$ is a domain; by \cite[Lemma~11.23]{atiyah_macdonald} $R$ is a domain and we conclude.
	
	(3) Let $P$ be a homogeneous prime of $A$. Consider the short exact sequence
	\[
	0 \to P \to A \to A/P \to 0.
	\]
	Its $M$-adic completion is the short exact sequence
	\[
	0 \to PR \to R \to R/PR \to 0.
	\]
	Since $R/PR$ is the completion of the domain $A/P$ at its maximal homogeneous ideal, by (2) we have that $R/PR$ is a domain, hence $PR$ is a prime of $R$.
	It remains to prove that $PR \cap A = P$.
	
	Since $A_M \into R$ is faithfully flat, by \cite[Exercise~3.16]{atiyah_macdonald} we have $PR \cap A_M = (PA_M)R \cap A_M = P A_M$. By contracting this to $A$ we get $PR \cap A = P A_M \cap A = P$ and we conclude.
	
	(4) is an easy consequence of (3) and of the well known fact that says that the minimal primes of $A$ are homogeneous.
	We leave the proof to the reader.
\end{proof}

\subsection{Segre products} \label{sec:segre_products}
This section can be omitted if one is interested in the proof of Theorems~\ref{thm:main_toric_singularity}, \ref{thm:main_no_murphys_law_for_toric_singularities} and \ref{thm:main_hull_K-moduli_stack} only; it will be be used only in the proof of Theorem~\ref{thm:main_many_branches} in \S\ref{sec:arbitrarily_many_branches}.

Fix a field $\KK$. We consider the following construction dating back to \cite{chow_unmixedness}:

\begin{definition}
	If $A = \bigoplus_{n \geq 0} A_n$ and $B = \bigoplus_{n \geq 0} B_n$ are two standard graded $\KK$-algebras, then their \emph{Segre product} is the standard graded $\KK$-algebra $A \# B := \bigoplus_{n \geq 0} A_n \otimes_\KK B_n$.
\end{definition}

The Segre product is a direct summand of the tensor algebra $A \otimes_\KK B$.

\begin{lemma} \label{lem:irr_comp_segre_product}
	Let $\KK$ be an algebraically closed field.
	Let $A$ and $B$ be two standard graded $\KK$-algebras.
	If $A$ has $m$ minimal primes and $B$ has $n$ minimal primes, then the Segre product $A \# B$ has $mn$ minimal primes.
\end{lemma}

\begin{proof}
	The minimal primes of $A$ are $1$-to-$1$ correspondence with the irreducible components of $\Proj A$. The same holds for $B$ and $A \# B$.
	Since $\KK$ is algebraically closed, the fibred product over $\KK$ of two irreducible schemes of finite type over $\KK$ is irreducible. So we conclude with the following observation:
	there exists a natural isomorphism of $\KK$-schemes
	\[
	\Proj A \# B \simeq \Proj A \times_{\Spec \KK} \Proj B;
	\]
	this is obtained by gluing the isomorphisms of affine schemes induced by the $\KK$-algebra isomorphisms
	\[
	(A \# B)_{(f \otimes g)} \simeq A_{(f)} \otimes_\KK B_{(g)}
	\]
	for $f \in A_1$ and $g \in B_1$.
\end{proof}

\begin{remark} \label{rmk:segre_product_and_torus_actions}
Let $\KK$ be a field and let $A$ be a standard graded $\KK$-algebra.
The $\NN$-grading of $A$ gives an action of the torus $\Gm = \Spec \KK [t^{\pm}]$ on the $\KK$-scheme $\Spec A$.
Consider $\Spec A \times_{\Spec \KK} \Spec A = \Spec (A \otimes_\KK A)$ equipped with the $\Gm$-action given by the 
action above on the first factor and the inverse action on the second factor; in other words we consider the $\ZZ$-grading of $A \otimes_\KK A$ given by $\deg (a \otimes b) = \deg a - \deg b$, for $a$ and $b$ homogeneous elements of $A$.
The subalgebra of $A \otimes_\KK A$ made up of the elements of degree $0$ is exactly the Segre product $A \# A$.
This shows that $\Spec (A \# A)$ is the quotient of $\Spec (A \otimes_\KK A)$ under the action of $\Gm$.
\end{remark}

\subsection{Generalised Newton identities} \label{sec:newton}

In what follows, $n$ denotes a positive integer.

\begin{lemma} \label{lem:newton_1}
	Let $A$ be a ring. Fix $a_1, \dots, a_n$ and $x_1, \dots, x_n$ in $A$.
	For each integer $k \geq 1$, consider
	\[
	\alpha_k = \sum_{i=1}^n a_i x_i^k \in A.
	\]
	Then the ideals
	\begin{equation*}
	\left( \alpha_k \mid k \geq 1 \right) \qquad \text{and} \qquad \left( \alpha_k \mid 1 \leq k \leq n \right)
	\end{equation*}
	of $A$ coincide.
\end{lemma}

\begin{proof}
	We will prove a generalised version of Newton's identities of symmetric functions.
	Set
	\begin{equation*}
	s_r = (-1)^r \sum_{1 \leq i_1 < \dots < i_r \leq n} x_{i_1} \cdots x_{i_r}
	\end{equation*}	
	for each $r = 1, \dots, n$.
	Consider the polynomial
	\begin{equation*}
	f(t) = \prod_{i=1}^{n} (t-x_i) = t^n + s_1 t^{n-1} + \cdots + s_0 \in A[t].
	\end{equation*}
	Obviously we have the equalities
	\begin{gather*}
		x_1^n + s_1 x_1^{n-1} + \cdots + s_n = f(x_1) = 0, \\
		\vdots \\
		x_n^n + s_1 x_n^{n-1} + \cdots + s_n = f(x_n) = 0.
	\end{gather*}
	
	Now fix $k>n$. We multiply the $i$th equality above by $a_i x_i^{k-n}$ and get
	\begin{gather*}
		a_1 x_1^k + s_1 a_1 x_1^{k-1} + \cdots + s_n a_1 x_1^{k-n}  = 0, \\
		\vdots \\
		a_n x_n^k + s_1 a_n x_n^{k-1} + \cdots + s_n a_n x_n^{k-n}  = 0. 
	\end{gather*}
	Adding these equalities together, we obtain
	\begin{equation*}
	\alpha_k + s_1 \alpha_{k-1} + \cdots + s_n \alpha_{k-n} = 0.
	\end{equation*}	
	This implies that, for each $k>n$, $\alpha_k$ is in the ideal $(\alpha_{k-1}, \dots, \alpha_{k-n})$.
	With an obvious inductive argument, we have that $\alpha_k$ is in the ideal $(\alpha_n, \dots, \alpha_1)$, for each $k>n$.
\end{proof}

\begin{lemma} \label{lem:newton_2}
Let $\KK$ be a field and let $A$ be a $\KK$-algebra. Consider a $2 \times n$ matrix
\[
\begin{pmatrix}
a_1 & \cdots & a_n \\
b_1 & \cdots & b_n
\end{pmatrix}
\]
with entries in $\KK$ and of rank $2$.
Fix $x_1, \dots, x_n \in A$ and,
for each integer $k \geq 1$, consider
\[
\alpha_k = \sum_{i=1}^n a_i x_i^k \in A \quad \text{and} \quad \beta_k = \sum_{i=1}^n b_i x_i^k \in A.
\]
Then the ideals
\begin{equation*}
\left( \alpha_k, \beta_k \mid k \geq 1 \right) \qquad \text{and} \qquad \left( \alpha_k, \beta_k \mid 1 \leq k \leq n-1 \right)
\end{equation*}
of $A$ coincide.
\end{lemma}

\begin{proof}
Let $I$ be the ideal on the right.
By Lemma~\ref{lem:newton_1} it is enough to show that $\al_n$ and $\beta_n$ lie in $I$.
For each $k \geq 0$ consider
\[
b_n \al_k - a_n \beta_k = \sum_{i=1}^{n-1} (b_n a_i - a_n b_i) x_i^k.
\]
By Lemma~\ref{lem:newton_1} applied to $b_n a_1 - a_n b_1, \dots, b_n a_{n-1} - a_n b_{n-1}$ and to $x_1, \dots, x_{n-1}$, we have that
\[
b_n \al_n - a_n \beta_n \in \left( b_n \al_{1} - a_n \beta_1, \dots, b_n \al_{n-1} - a_n \beta_{n-1} \right) \subseteq I.
\]
In a completely analogous way, we prove that
$
b_j \al_n - a_j \beta_n \in I
$,
for every $j=1,\dots,n$.
Since the matrix in the statement has rank $2$, there exist two indices $j$ and $h$ such that $b_j a_h - a_j b_h \neq 0$.
From $
b_j \al_n - a_j \beta_n \in I
$
and
$
b_h \al_n - a_h \beta_n \in I
$ we deduce that $\al_n \in I$ and $\beta_n \in I$.
\end{proof}

\begin{remark} \label{rmk:newton_algebra}
Let $\KK$ is a field and let
\[
\begin{pmatrix}
a_1 & \cdots & a_n \\
b_1 & \cdots & b_n
\end{pmatrix}
\]
be a $2 \times n$ matrix
with entries in $\KK$ and of rank $2$.
Consider the polynomial ring $S = \KK[x_1, \dots, x_n]$ and the ideal $I \subseteq S$ generated by
\[
\alpha_k = \sum_{i=1}^n a_i x_i^k  \quad \text{and} \quad \beta_k = \sum_{i=1}^n b_i x_i^k \qquad \text{for } k \geq 1.
\]
By Lemma~\ref{lem:newton_2} $I$ is generated by $\al_1, \beta_1, \dots, \al_{n-1}, \beta_{n-1}$.
The quotient $S/I$ is a standard graded $\KK$-algebra.

The ideal $I$ does not change if we multiply the matrix above by a matrix in $\GL_2(\KK)$ on the left.
The isomorphism class of $S/I$ does not change if we permute the columns of the matrix above.
This implies that the isomorphism class of $S/I$ depends only on the point in $\mathrm{Gr}(2, \KK^n) / \mathfrak{S}_n$, which is the quotient of the Grassmannian $\mathrm{Gr}(2, \KK^n) $ under the action of the symmetric group $\mathfrak{S}_n$.
\end{remark}

Here we consider a slight variation of the construction in Remark~\ref{rmk:newton_algebra}.

\begin{lemma} \label{lem:removing_one_indeterminate}
	Let $\KK$ be a field and consider a $2 \times m$ matrix
\[
\begin{pmatrix}
a_1 & \cdots & a_m \\
b_1 & \cdots & b_m
\end{pmatrix}
\]
with entries in $\KK$ and such that the sum of the columns is the zero vector.
For each $j=1, \dots, m$, consider the polynomial ring $S_j = \KK[x_1, \dots, \widehat{x_j}, \dots, x_m]$ and the ideal $I_j \subseteq S_j$ generated by
\[
\alpha_{j,k} = \sum_{\substack{1 \leq i \leq m \\ i \neq j}}   a_i x_i^k \quad \text{and} \quad \beta_{j,k} =\sum_{\substack{1 \leq i \leq m \\ i \neq j}} b_i x_i^k
\]
for all $k \geq 1$.
Then the standard graded $\KK$-algebras $S_1 / I_1$, \dots, $S_m / I_m$ are all isomorphic.
\end{lemma}

\begin{proof}
It is enough to show that $S_1 / I_1$ and $S_m / I_m$ are isomorphic.
To avoid confusion, we denote the indeterminates of $S_m$ as $y_1, \dots, y_{m-1}$.
Consider the $\KK$-linear isomorphism $\phiv \colon S_1 \to S_m$ given by $x_2 \mapsto y_2 - y_1$, \dots, $x_{m-1} \mapsto y_{m-1} - y_1$, $x_m \mapsto -y_1$.
Then for every $k \geq 1$
\begin{align*}
\phiv(\al_{1,k}) &= \phiv \left(\sum_{i=2}^m a_i x_i^k  \right) = \sum_{i=2}^{m-1} a_i (y_i-y_1)^k + a_m (-y_1)^k \\
&= \sum_{i=2}^{m-1} a_i      \sum_{l=0}^k {k \choose l}     y_i^l (-y_1)^{k-l}      + a_m (-y_1)^k \\
&= \sum_{l=0}^k {k \choose l} (-y_1)^{k-l} \left( \al_{m,l} - a_1 y_1^l   \right) + a_m (-y_1)^k \\
&= \sum_{l=0}^k {k \choose l} (-y_1)^{k-l} \al_{m,l} + a_m (-y_1)^k \\
&= \sum_{l=1}^k {k \choose l} (-y_1)^{k-l} \al_{m,l} \in I_m.
\end{align*}
Notice that in the last equality we have used $a_1 + \cdots + a_m = 0$. In a completely analogous way we can prove $\phiv(\beta_{1,k}) \in I_m$.
Therefore $\phiv(I_1) \subseteq I_m$. By using the inverse of $\phiv$, one can show $\phiv(I_1) = I_m$.
  \end{proof}

\section{Deformations of isolated toric singularities}
\label{sec:toric_singularities}
\subsection{From polygons to algebras}

Here we consider a construction, due to Altmann~\cite{altmann_versal}, of a standard graded algebra associated to a polygon.
Here $\CC$ denotes the field of complex numbers (see Remark~\ref{rmk:field}).

\begin{construction}[{Altmann~\cite{altmann_versal}}] \label{constr:algebra_A_F}
Let $F$ be a polygon in $\RR^2$ with $m$ edges.
Fix an orientation on $\RR^2$ and a compatible ordering of the edges of $F$: $E_1, \dots, E_m$.
For each $i=1,\dots,m$, let
\[
\begin{pmatrix}
a_i \\ b_i
\end{pmatrix}
\in \RR^2
\]
be the vector associated to the edge $E_i$, i.e.\ the difference between the second vertex of $E_i$ and the first vertex of $E_i$.
In the polynomial ring $\CC[x_1, \dots, x_{m-1}]$ we consider the homogeneous ideal $I_F$ generated by
\[
\sum_{i=1}^{m-1} a_i x_i^k \quad \text{and} \quad \sum_{i=1}^{m-1} b_i x_i^k \quad \text{for } k \geq 1.
\]
We consider the standard graded $\CC$-algebra
\[
A_F := \CC[x_1, \dots, x_{m-1}] / I_F.
\]
\end{construction}

\begin{remark}
Since
\[
\begin{pmatrix}
a_m \\ b_m
\end{pmatrix}
= -
\sum_{i=1}^{m-1}
\begin{pmatrix}
a_i \\ b_i
\end{pmatrix},
\]
by Lemma~\ref{lem:removing_one_indeterminate}, the isomorphism class of the graded $\CC$-algebra $A_F$ depends neither on the ordering of the edges of $F$ nor on the orientation of $\ZZ^2$.

If we apply an affine transformation to $F$, the matrix whose columns are the edges of $F$ is multiplied on the left by an element of $\GL_2(\RR)$. Therefore, by Remark~\ref{rmk:newton_algebra}, the isomorphism class of $A_F$ 
 depends only on the $\GL_2(\RR) \ltimes \RR^2$-equivalence class of the polygon $F$.

Since $F$ has dimension $2$, the matrix
\begin{equation*}
\begin{pmatrix}
	a_1 & \cdots & a_{m-1} \\
	b_1 & \cdots & b_{m-1}
\end{pmatrix}
\end{equation*}
has rank $2$; therefore, by Lemma~\ref{lem:newton_2}, the ideal $I_F$ is generated by
\[
\sum_{i=1}^{m-1} a_i x_i^k \quad \text{and} \quad \sum_{i=1}^{m-1} b_i x_i^k \quad \text{for } 1 \leq k \leq m-2.
\]
\end{remark}

\begin{example} \label{ex:klaus}
	In \cite[\S9]{altmann_versal} there are the following three examples.
If $F$ is the quadrilateral with vertices $(1,1)$, $(-1,0)$, $(-1,-1)$, $(0,-1)$, then $A_F$ is isomorphic to $\CC[x] / (x^2)$.
If $F$ is the pentagon with vertices $(1,0)$, $(0,1)$, $(-1,1)$, $(-1,0)$, $(0,-1)$, then $A_F$ is isomorphic to $\CC[x,y]/(x^2,xy)$.
If $F$ is the hexagon with vertices $(1,0)$, $(1,1)$, $(0,1)$, $(-1,0)$, $(-1,-1)$, $(0,-1)$, then $A_F$ is isomorphic to $\CC[x,y,z]/(xy,xz)$.
\end{example}

\begin{proposition} \label{prop:triangles_and_quadrilaterals}
	Let $F$ be a polygon in $\RR^2$.
	\begin{enumerate}[label=(\roman*)]
		\item If $F$ is a triangle, then $A_F$ is isomorphic to $ \CC$.
		\item If $F$ is a parallelogram, then $A_F$ is isomorphic to $\CC[x]$.
		\item If $F$ is a quadrilateral but not a parallelogram, then $A_F$ is isomorphic to $\CC[x] / (x^2)$.
	\end{enumerate}
\end{proposition}

\begin{proof}
	(i) Up to affine equivalence we can assume that $F$ is the convex hull of $(0,0)$, $(1,0)$ and $(1,1)$.
	We consider the $2 \times 2$ matrix whose columns are the first $2$ edges of $F$:
	\[
	\begin{matrix}
	x_1 & x_2 \\
	\hline
	1 & 0 \\
	0 & 1
	\end{matrix}
	\]
	where we have used the indeterminates $x_1$ and $x_2$ to label the edges.
	The ideal $I_F \subseteq \CC[x_1, x_2]$ is generated by $x_1$ and $x_2$.
	
\begin{figure}
	\centering
	\includegraphics[width=0.5\linewidth]{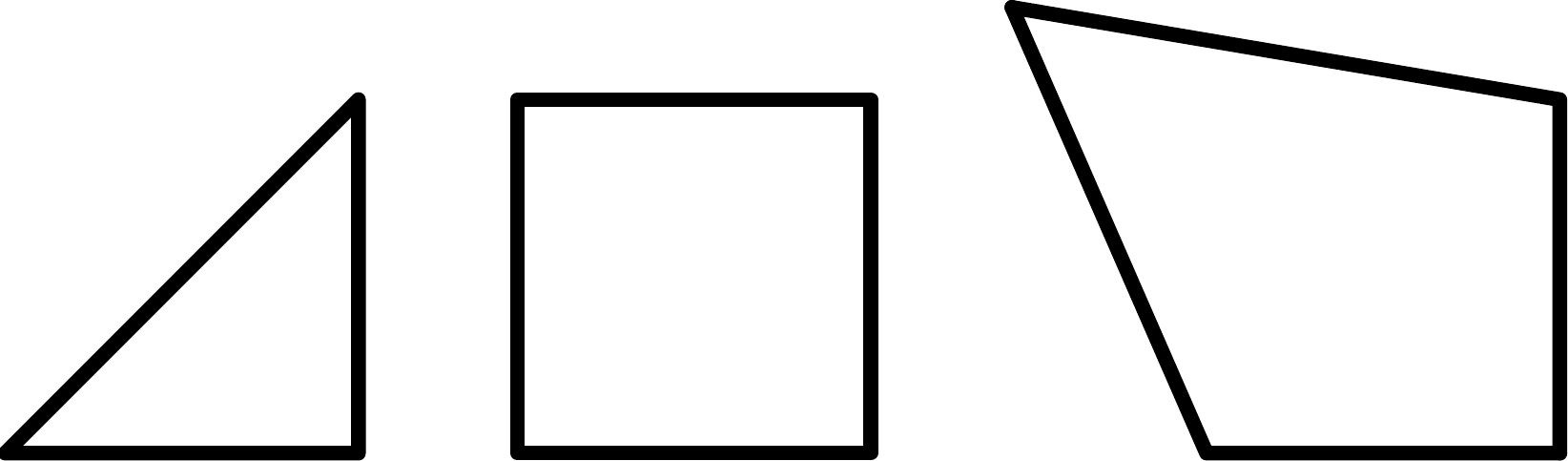}
	\caption{The polygons considered in Proposition~\ref{prop:triangles_and_quadrilaterals}}
	\label{fig:triangolo_quadrilateri}
\end{figure}

	(ii) Up to affine equivalence we can assume that $F$ is the convex hull of $(0,0)$, $(1,0)$, $(1,1)$, $(0,1)$.
	We consider the $2 \times 3$ matrix whose columns are the first $3$ edges of $F$
	\[
	\begin{matrix}
	p & q & x \\
	\hline
	1 & 0 & -1 \\
	0 & 1 & 0
	\end{matrix}
	\]
	where we have used the indeterminates $p,q,x$ to label the edges.
	The ideal $I_F \subseteq \CC[p,q,x]$ is generated by $p-x$, $p^2 - x^2$, $q$, and $q^2$.

	(iii) Up to affine equivalence we can assume that $F$ is the convex hull of $(0,0)$, $(1,0)$, $(1,1)$, $(c,d)$, where $c,d \in \RR$ are such that $c < 1$, $d>0$ and $(c,d) \neq (1,1)$.
	We consider the $2 \times 3$ matrix whose columns are the first $3$ edges of $F$
	\[
	\begin{matrix}
	p & q & x \\
	\hline
	1 & 0 & c-1 \\
	0 & 1 & d-1
	\end{matrix}
	\]
	where we have used the indeterminates $p,q,x$ to label the edges.
	
	The ideal $I_F \subseteq \CC[p,q,x]$ is generated by $p+(c-1)x$, $p^2+(c-1)x^2$, $q + (d-1)x$, and $q^2 + (d-1)x^2$.	
	Therefore $A_F = \CC[p,q,x] / I_F$ is isomorphic to $\CC[x] / J$, where $J$ is the ideal generated by $p^2+(c-1)x^2$ and $q^2 + (d-1)x^2$, where $p = (1-c) x$ and $q = (1-d) x$.
	Now
	$p^2 + (c-1) x^2 = c(c - 1) x^2$
	and
	$q^2 + (d-1) x^2 = d (d-1) x^2$.
	By the hypotheses on $c$ and $d$, we have $J = (x^2)$.
	\end{proof}

\begin{proposition} \label{prop:hilbert_function_A_F}
	Let $F$ be a polygon in $\RR^2$ with $m$ vertices and let $H \colon \NN \to \NN$ be the Hilbert function of $A_F$. We have $H(1) = m-3$ and, if $m \geq 5$, then
	\[
	H(2) =  \frac{m^2 - 5m + 2}{2}.
	\]
\end{proposition}

\begin{proof}
The first assertion follows from the fact that among the generators of the ideal $I_F$ there are two linear polynomials which are linearly independent. Therefore $\dim_\CC (I_F)_1 = 2$. So $\dim_\CC (A_F)_1 = (m-1) - 2$.

Now we prove the second assertion, so assume $m \geq 5$.
Since $F$ is not a parallelogram, there exist $3$ consecutive edges which are pairwise non-parallel.
Up to $\GL_2(\RR) \ltimes \RR^2$ we can assume that these $3$ edges are the following: the first edge goes from $(0,0)$ to $(1,0)$, the second edge goes from $(1,0)$ to $(1,1)$,
the third edge goes from $(1,1)$ to $(1+a_1, 1+b_1)$.
See Figure~\ref{fig:poligono_con_linea_rossa}.
As the third edge cannot be parallel to the second edge, $a_1 \neq 0$. Moreover, by convexity at $(1,1)$, one must have $a_1 < 0$.
Moreover, it is clear that $1+b_1 > 0$.
From the non-parallelism assumption between the first edge and the third edge, $b_1 \neq 0$.

\begin{figure}
	\centering
	\def\svgwidth{40mm}
	\medskip
\begingroup%
  \makeatletter%
  \providecommand\color[2][]{%
    \errmessage{(Inkscape) Color is used for the text in Inkscape, but the package 'color.sty' is not loaded}%
    \renewcommand\color[2][]{}%
  }%
  \providecommand\transparent[1]{%
    \errmessage{(Inkscape) Transparency is used (non-zero) for the text in Inkscape, but the package 'transparent.sty' is not loaded}%
    \renewcommand\transparent[1]{}%
  }%
  \providecommand\rotatebox[2]{#2}%
  \newcommand*\fsize{\dimexpr\f@size pt\relax}%
  \newcommand*\lineheight[1]{\fontsize{\fsize}{#1\fsize}\selectfont}%
  \ifx\svgwidth\undefined%
    \setlength{\unitlength}{354.35833142bp}%
    \ifx\svgscale\undefined%
      \relax%
    \else%
      \setlength{\unitlength}{\unitlength * \real{\svgscale}}%
    \fi%
  \else%
    \setlength{\unitlength}{\svgwidth}%
  \fi%
  \global\let\svgwidth\undefined%
  \global\let\svgscale\undefined%
  \makeatother%
  \begin{picture}(1,1.24492753)%
    \lineheight{1}%
    \setlength\tabcolsep{0pt}%
    \put(0,0){\includegraphics[width=\unitlength,page=1]{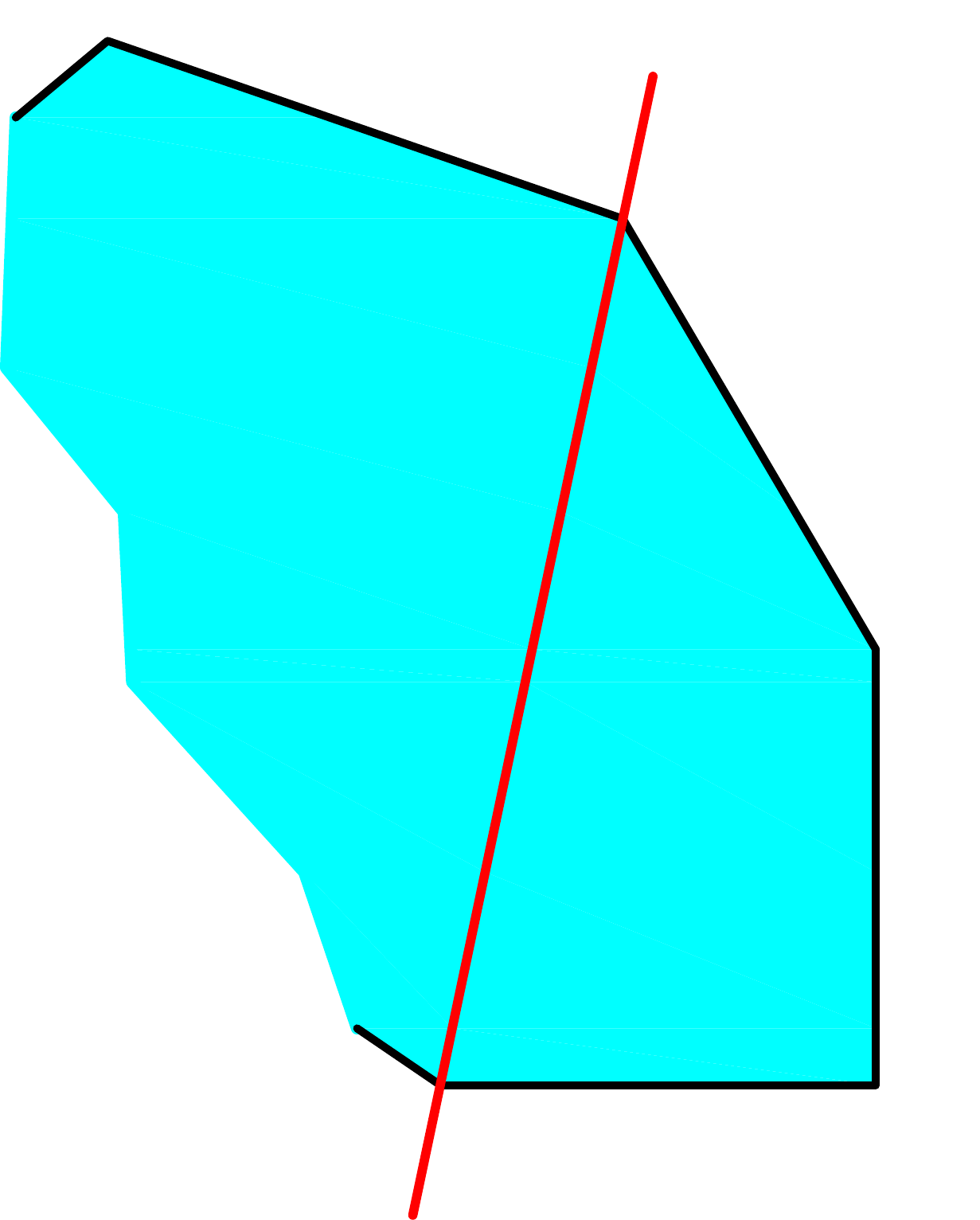}}%
    \put(0.450,0.060){\color[rgb]{0,0,0}\makebox(0,0)[lt]{\lineheight{1.25}\smash{\begin{tabular}[t]{l}$(0,0)$\end{tabular}}}}%
    \put(0.880,0.060){\color[rgb]{0,0,0}\makebox(0,0)[lt]{\lineheight{1.25}\smash{\begin{tabular}[t]{l}$(1,0)$\end{tabular}}}}%
    \put(0.905,0.58){\color[rgb]{0,0,0}\makebox(0,0)[lt]{\lineheight{1.25}\smash{\begin{tabular}[t]{l}$(1,1)$\end{tabular}}}}%
    \put(0.655,1.027){\color[rgb]{0,0,0}\makebox(0,0)[lt]{\lineheight{1.25}\smash{\begin{tabular}[t]{l}$(1+a_1,1+b_1)$\end{tabular}}}}%
    \put(0.040,1.225){\color[rgb]{0,0,0}\makebox(0,0)[lt]{\lineheight{1.25}\smash{\begin{tabular}[t]{l}$(1+a_1+a_2,1+b_1+b_2)$\end{tabular}}}}%
  \end{picture}%
\endgroup%

	\caption{}
	\label{fig:poligono_con_linea_rossa}
\end{figure}

Let $\ell$ be the line passing through $(0,0)$ and $(1+a_1, 1+b_1)$, which is depicted in red in Figure~\ref{fig:poligono_con_linea_rossa}.
Now consider the fourth edge, which goes from $(1+a_1, 1+b_1)$ to $(1+a_1+a_2, 1+b_1+b_2)$.
It is clear that the vertices $(1+a_1+a_2, 1+b_1+b_2)$ and $(1,0)$ lie on  different sides with respect to $\ell$; this easily implies the inequality
\begin{equation*}
a_1 b_2 - a_2 b_1 - a_2 + b_2 > 0.
\end{equation*}

We consider the $2 \times (m-1)$ matrix
\[
\begin{matrix}
p & q & x_1 & x_2 & \cdots & x_{m-3} \\
\hline
1 & 0 & a_1 & a_2 & \cdots & a_{m-3} \\
0 & 1 & b_1 & b_2 & \cdots & b_{m-3}
\end{matrix}
\]
and the $\CC$-algebra $A_F = \CC[p,q,x_1, \dots, x_{m-3}] / I_F$. The polynomials
\begin{gather*}
	p + a_1 x_1 + a_2 x_2 + \cdots + a_{m-3} x_{m-3} \\
	p^2 + a_1 x_1^2 + a_2 x_2^2 + \cdots + a_{m-3} x_{m-3}^2 \\
			q + b_1 x_1 + b_2 x_2 + \cdots + b_{m-3} x_{m-3} \\
		q^2 + b_1 x_1^2 + b_2 x_2^2 + \cdots + b_{m-3} x_{m-3}^2
	\end{gather*}
are contained in $I_F$.
Then $A_F$ is isomorphic to $\CC[x_1, x_2, \dots, x_{m-3}] / J$, where $J$ is a certain homogeneous ideal such that the degree $1$ part of $J$ is zero and the degree $2$ part of $J$ is spanned by the following two quadrics:
\begin{align*}
f &= (-a_1 x_1 - a_2 x_2 - \cdots - a_{m-3} x_{m-3})^2 + a_1 x_1^2 + a_2 x_2^2 + \cdots + a_{m-3} x_{m-3}^2 \\
&= a_1(a_1 + 1) x_1^2 + 2 a_1 a_2 x_1 x_2 + a_2(a_2 + 1) x_2^2 + \cdots, \\
g &= (-b_1 x_1 - b_2 x_2 - \cdots - b_{m-3} x_{m-3})^2 + b_1 x_1^2 + b_2 x_2^2 + \cdots + b_{m-3} x_{m-3}^2 \\
&= b_1(b_1 + 1) x_1^2 + 2 b_1 b_2 x_1 x_2 + b_2(b_2 + 1) x_2^2 + \cdots.
\end{align*}
Since
\begin{equation*}
\det \begin{pmatrix}
	a_1 (a_1 + 1) & 2 a_1 a_2 \\
		b_1 (b_1 + 1) & 2 b_1 b_2
\end{pmatrix}
= 2 a_1 b_1 (a_1 b_2 - a_2 b_1 - a_2 + b_2) \neq 0,
\end{equation*}
the quadrics $f$ and $g$ are linearly independent. This implies $\dim_\CC J_2 = 2$, therefore $\dim_\CC (A_F)_2 = \frac{(m-3)(m-2)}{2}-2$.
\end{proof}

\begin{lemma} \label{lem:two_quadrics_in_P^1}
	Let $f,g \in \CC[x,y]$ be two non-zero homogeneous polynomials of degree $2$
	 such that they are coprime.
	Then $(x,y)^3 \subseteq (f,g)$  and $\CC[x,y] / (f,g)$ is isomorphic to $\CC[x,y] / (x^2, y^2)$ as standard graded $\CC$-algebras.
\end{lemma}

\begin{proof}
	The ideal containment follows from the graded $\CC$-algebra isomorphism between $\CC[x,y] / (f,g)$ and $\CC[x,y] / (x^2, y^2)$, because the isomorphism will be given by a linear change of coordinates. Therefore it is enough to construct the graded isomorphism between these two $\NN$-graded $\CC$-algebras.

	Since $f$ and $g$ are coprime, they are linearly independent and their zero loci $\rV(f)$ and $\rV(g)$ in $\PP^1$ are disjoint.
	With a linear change of coordinates we can assume $[0:1] \in \rV(f)$ and $[1:0] \in \rV(g)$.
	Thus, up to multiplicative constants, $f = x(x-\al y)$ and $g = y(y-\beta x)$ for $\al, \beta \in \CC$ with $\al \beta \neq 1$.
	Notice that $\rV(f) = \{ [0:1], [\al : 1] \}$ and $\rV(g) = \{[1:0], [1:\beta]\}$.
	
	For $[\lambda : \mu] \in \PP^1$ consider the form $\lambda f + \mu g = \lambda x^2 - (\lambda \al + \mu \beta)xy + \mu y^2$. Its discriminant is $\Delta = (\lambda \al + \mu \beta)^2 - 4 \lambda \mu = \al^2 \lambda^2 + 2 (\al \beta -2) \lambda \mu + \beta^2 \mu^2$.
	The discriminant of $\Delta$ is $(\al \beta - 2)^2 - \al^2 \beta^2 = 4 (1-\al \beta) \neq 0$.
	Therefore $\Delta$ has  two simple zeroes in $\PP^1$. In other words, the pencil spanned by $f$ and $g$ contains exactly $2$ non-reduced forms $h_1$ and $h_2$.
	With a linear change of coordinates we can assume that $h_1 = x^2$ and $h_2 = y^2$.
\end{proof}

\begin{example} \label{ex:pentagono_indecomponibile}

		\begin{figure}
		\centering
		\includegraphics[width=0.2\linewidth]{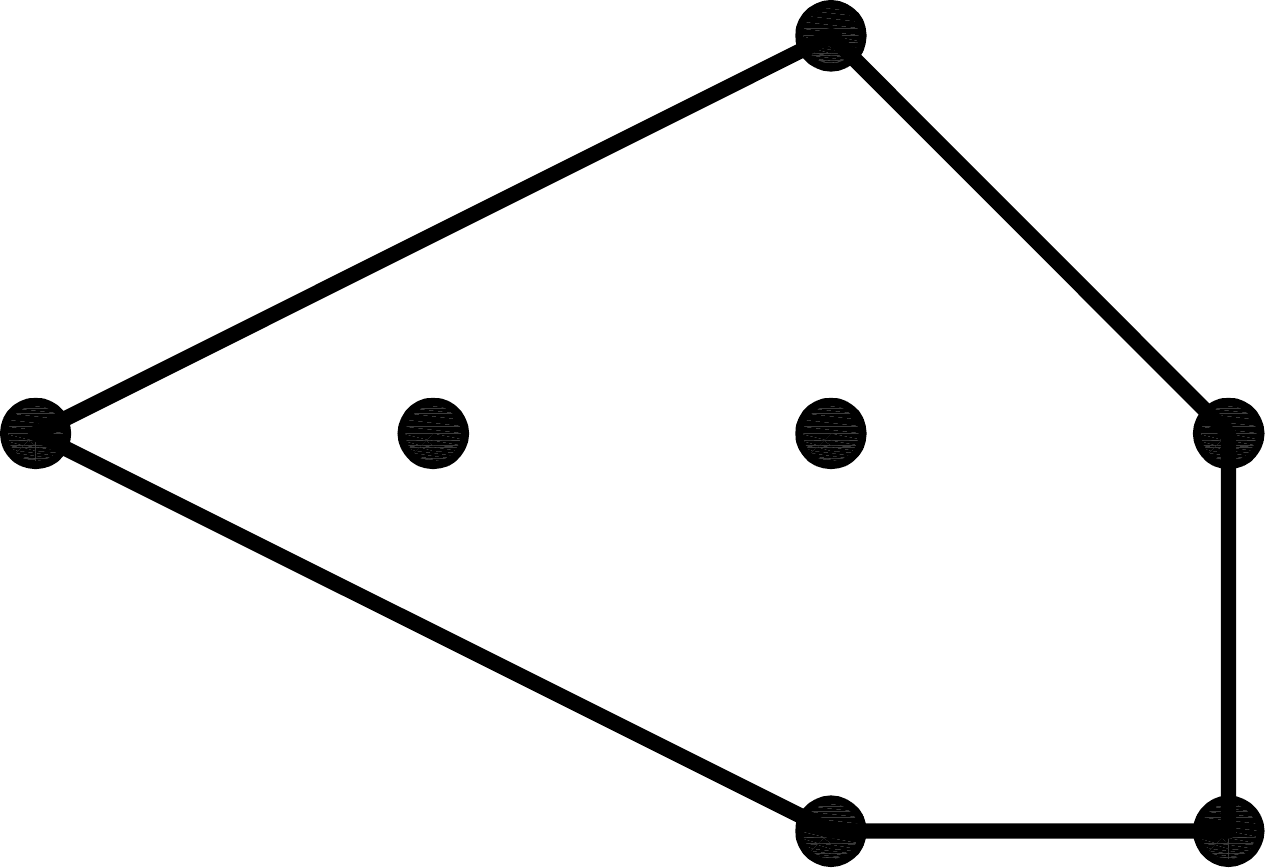}
		\caption{The pentagon in Example~\ref{ex:pentagono_indecomponibile}}
		\label{fig:pentagono_indecomponibile}
	\end{figure}
	
	Let $F$ be the lattice pentagon depicted in Figure~\ref{fig:pentagono_indecomponibile}.
	The first $4$ edges of $F$, starting from the bottom, are the columns of
	\begin{equation*}
	\begin{matrix}
	p & q & x & y \\
	\hline
	1 & 0 & -1 & -2 \\
	0 & 1 & 1 & -1
	\end{matrix}
	\end{equation*}
	and the algebra $A_F$ is isomorphic to $\CC[x,y]/J$, where $J$ is the ideal generated by $f = p^2 - x^2 - 2 y^2$, $g = q^2 + x^2 - y^2$,
	$p^3 - x^3 - 2 y^3$, $q^3 + x^3 - y^3$, where $p = x+2y$ and $q = -x+y$. One has $f = 2xy + y^2$ and $g = x^2 - xy$.
	Since $f$ and $g$ are coprime, by Lemma~\ref{lem:two_quadrics_in_P^1} we have $J = (f,g)$ and that $A_F$ is isomorphic to $\CC[x,y]/(x^2, y^2)$.
\end{example}

\begin{example} \label{ex:pentagono_strano}
	
		\begin{figure}
		\centering
		\includegraphics[width=0.33\linewidth]{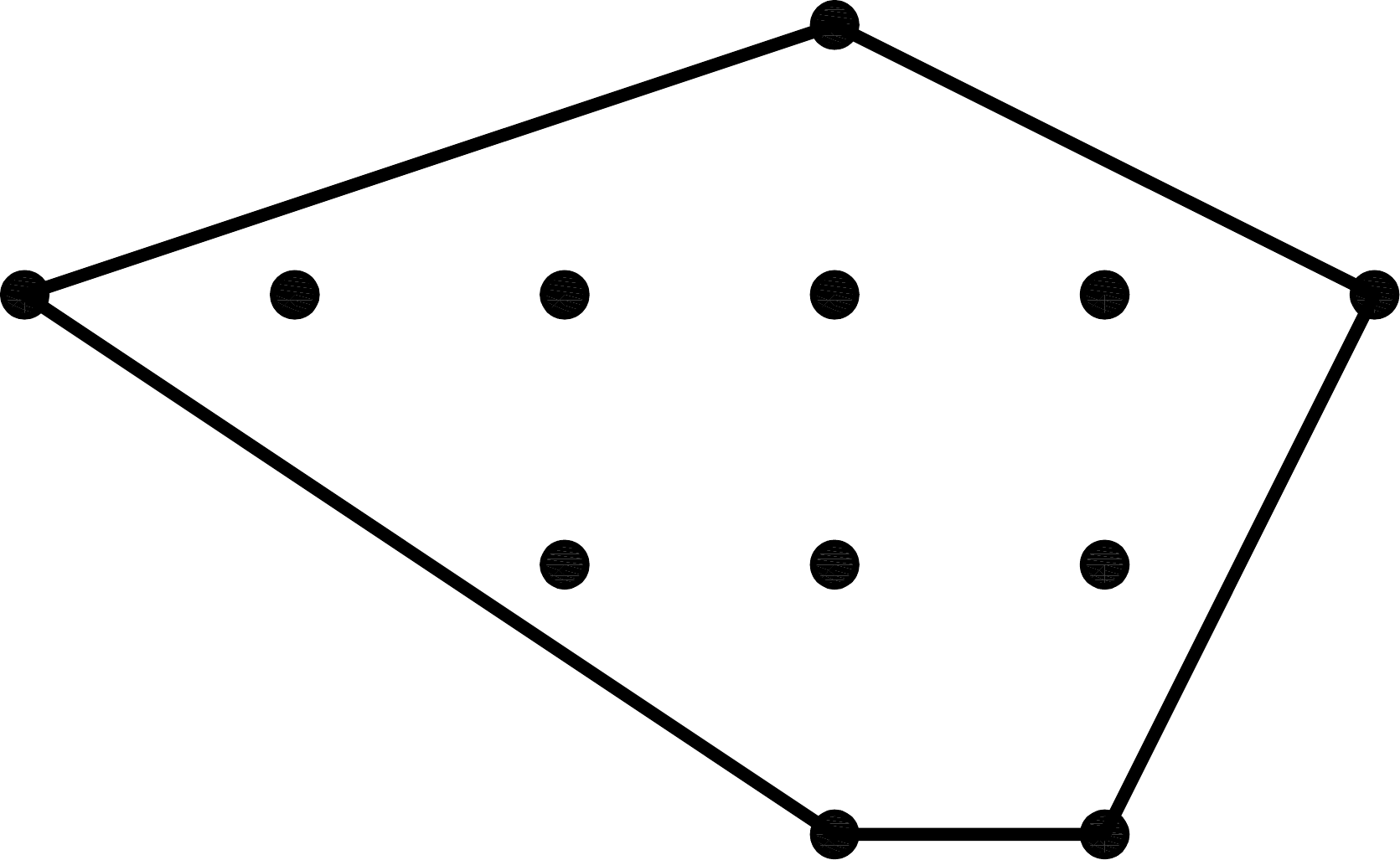}
		\caption{The pentagon in Example~\ref{ex:pentagono_strano}}
		\label{fig:pentagono_strano}
	\end{figure}
	
	Let $F$ be the lattice pentagon depicted in Figure~\ref{fig:pentagono_strano}.
	The first $4$ edges of $F$, starting from the bottom, are the columns of
	\begin{equation*}
		\begin{matrix}
			p & q & u & v \\
			\hline
			1 & 1 & -2 & -3 \\
			0 & 2 & 1 & -1
		\end{matrix}
	\end{equation*}
and the algebra $A_F$ is isomorphic to $\CC[u,v]/J$, where $J$ is the ideal generated by $f_2 = p^2 + q^2 - 2u^2 - 3v^2$, $g_2 = 2q^2 + u^2 - v^2$, $f_3 = p^3 + q^3 - 2u^3 - 3v^3$, $g_3 = 2q^3 + u^3 - v^3$, where $p = - (q -2u-3v)$ and $q = - \frac{1}{2}(u-v)$. One has
\begin{gather*}
	f_2 = \frac{9}{2} u^2 + 12uv + \frac{7}{2} v^2 = \frac{9}{2} \left( u + \frac{1}{3}v \right) \left( u + \frac{7}{3}v \right) \\
	g_2 = \frac{3}{2} u^2 - uv - \frac{1}{2} v^2 = \frac{3}{2} \left( u + \frac{1}{3}v \right) \left( u-v \right).
	\end{gather*} 
So $f_2$ and $g_2$ are not coprime.
With the linear change of coordinates given by $x = u + \frac{1}{3}v$ and $y = u$, one sees that $A_F$ is isomorphic to $\CC[x,y] / (x^2, xy, y^3)$.

\end{example}

\begin{proposition} \label{prop:pentagon}
	If $F$ is a pentagon in $\RR^2$, then
	$A_F$ is isomorphic to one of the following standard graded  $\CC$-algebras:
	\[
	\CC[x,y]/(x^2, y^2), \quad \CC[x,y]/(x^2, xy, y^3), \quad \CC[x,y]/(x^2, xy).
	\] 
\end{proposition}

\begin{proof}
We know that $I_F$ is generated by two linear forms, two degree $2$ forms, and two degree $3$ forms.
As in the proof of Proposition~\ref{prop:hilbert_function_A_F} we have that $A_F$ is isomorphic to $\CC[x,y] / J$, where $J$ is a homogeneous ideal generated by two linearly independent quadrics $f, g$ and by two degree $3$ forms.
If $f$ and $g$ are coprime, then we conclude by Lemma~\ref{lem:two_quadrics_in_P^1}.

Assume that $f$ and $g$ are not coprime.
With a linear change of coordinates we can assume that the greatest common divisor between $f$ and $g$ is $x$.
So $f = xh_1$ and $g = x h_2$, where $h_1$ and $h_2$ are some degree $1$ forms.
By replacing $f$ and $g$ with appropriate linear combinations with coefficients in $\CC$, we can assume $f = x^2$ and $g = xy$. Then
\[
\CC[x,y]_3 = \mathrm{span}_\CC \left( x^3, x^2 y, x y^2, y^3 \right) \supseteq J_3 \supseteq \mathrm{span}_\CC \left( x^3, x^2 y, x y^2 \right).
\]
Therefore there are two cases depending on whether $J \ni y^3$ or not.
\end{proof}

\begin{remark}
By Proposition~\ref{prop:triangles_and_quadrilaterals} and Proposition~\ref{prop:pentagon} it is natural to ask whether $A_F$ is always a monomial algebra for every polygon $F$. This is not the case: for instance see Example~\ref{ex:esagono_sghembo}.
\end{remark}

\begin{remark} \label{rmk:field}
	All the results in this section remain valid if the field of complex numbers is replaced by any algebraically closed field $\KK$ which contains the coordinates of the edges of the polygon $F$: the linear parts of the affine transformations used in the proofs are in $\GL_2(\KK \cap \RR)$.
Starting from the next section we will be using lattice polygons only, therefore we can work over an arbitrary algebraically closed field of characteristic $0$, which is denoted by $\CC$ for simplicity.
\end{remark}


\subsection{Deformations of isolated Gorenstein toric $3$-fold singularities after Altmann}
\label{sec:altmann}

Every Gorenstein toric affine $3$-fold without torus factors arises in the following way from a lattice polygon:
if $F$ is a lattice polygon in $\ZZ^2$,  consider the cone $\RR_{\geq 0} (F \times \{1\})$ in the lattice $\ZZ^2 \oplus \ZZ = \ZZ^3$ and the corresponding toric affine variety.
This establishes a $1$-to-$1$ correspondence between isomorphism classes of Gorenstein toric affine $3$-folds without torus factors and $\GL_2(\ZZ) \ltimes \ZZ^2$-equivalence classes of lattice polygons in $\ZZ^2$.
Furthermore, in this situation, the singular locus has dimension at most $0$ if and only if $F$ has unit edges, i.e.\ if its edges have lattice length $1$.

If $F_0, \dots, F_r$ are lattice polytopes in $\ZZ^2$, then their \emph{Minkowski sum} is
\[
F_0 + \cdots + F_r := \{ v_0 + \cdots + v_r \mid v_0 \in F_0, \dots, v_r \in F_r \}.
\]
A Minkowski decomposition of a lattice polygon $F$ is the expression of $F$ as a Minkowski sum of lattice polytopes; of course we identify two Minkowski decompositions if they are the same after reordering the summands and after translation.
In Figure~\ref{fig:esagono_sghembo} two Minkowski decompositions of the same lattice hexagon are depicted.

Altmann~\cite{altmann_minkowski} has noticed that Minkowski decompositions of $F$ induce deformations of the toric affine variety associated to $F$ (see also \cite{ilten_vollmert, hochenegger_ilten, mavlyutov, petracci_mavlyutov}).
Moreover, he computes the miniversal deformation of an isolated Gorenstein toric $3$-fold singularity:

\begin{theorem}[{Altmann~\cite{altmann_versal}}] \label{thm:klaus_versal}
Let $F$ be a lattice polygon in $\ZZ^2$ with unit edges.
Let $X$ be the toric affine $\CC$-variety associated to the cone $\RR_{\geq 0} (F \times \{1\})$ in the lattice $\ZZ^2 \oplus \ZZ$.
Let $R$ be the hull of $\Def{X}$.
Then the following statements hold true.
\begin{enumerate}
	\item $R$ is the completion of the standard $\CC$-algebra $A_F$, introduced in Construction~\ref{constr:algebra_A_F}, at its maximal homogeneous ideal.
	\item The embedding dimension of $R$ is $m-3$, where $m$ is the number of edges of $F$.
	\item There is a canonical $1$-to-$1$ correspondence between minimal primes of $R$ and maximal Minkowski decompositions of $F$. Moreover, if a minimal prime $\frakp \subset R$ corresponds to the maximal Minkowski decomposition $F = F_0 + F_1 + \cdots + F_r$, then $\dim R / \frakp = r$.
	\item $R$ is artinian if and only if $F$ is Minkowski indecomposable.
\end{enumerate}	
\end{theorem}

The original formulation of Theorem~\ref{thm:klaus_versal}(1) in \cite{altmann_versal} is slightly different, but it is equivalent to what we have written above thanks to the proof of \cite[Lemma~3.4(1)]{altmann_versal}.

\begin{example}\label{ex:esagono_sghembo}
	Consider the lattice hexagon $F$ depicted in Figure~\ref{fig:esagono_sghembo}. The first $5$ edges of $F$, starting from the bottom and going anticlockwise, are the columns of
		\begin{figure}
		\centering
		\def\svgwidth{120mm}
		\medskip
\begingroup%
  \makeatletter%
  \providecommand\color[2][]{%
    \errmessage{(Inkscape) Color is used for the text in Inkscape, but the package 'color.sty' is not loaded}%
    \renewcommand\color[2][]{}%
  }%
  \providecommand\transparent[1]{%
    \errmessage{(Inkscape) Transparency is used (non-zero) for the text in Inkscape, but the package 'transparent.sty' is not loaded}%
    \renewcommand\transparent[1]{}%
  }%
  \providecommand\rotatebox[2]{#2}%
  \newcommand*\fsize{\dimexpr\f@size pt\relax}%
  \newcommand*\lineheight[1]{\fontsize{\fsize}{#1\fsize}\selectfont}%
  \ifx\svgwidth\undefined%
    \setlength{\unitlength}{483.17880362bp}%
    \ifx\svgscale\undefined%
      \relax%
    \else%
      \setlength{\unitlength}{\unitlength * \real{\svgscale}}%
    \fi%
  \else%
    \setlength{\unitlength}{\svgwidth}%
  \fi%
  \global\let\svgwidth\undefined%
  \global\let\svgscale\undefined%
  \makeatother%
  \begin{picture}(1,0.21553972)%
    \lineheight{1}%
    \setlength\tabcolsep{0pt}%
    \put(0,0){\includegraphics[width=\unitlength,page=1]{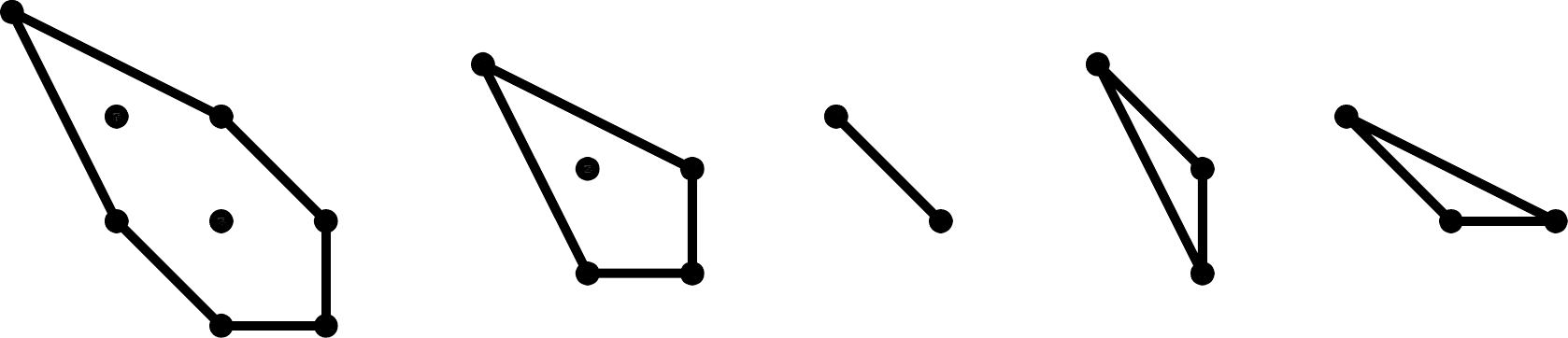}}%
    \put(0.2497,0.1032){\color[rgb]{0,0,0}\makebox(0,0)[lt]{\lineheight{1.25}\smash{\begin{tabular}[t]{l}$=$\end{tabular}}}}%
    \put(0.48,0.1032){\color[rgb]{0,0,0}\makebox(0,0)[lt]{\lineheight{1.25}\smash{\begin{tabular}[t]{l}$+$\end{tabular}}}}%
    \put(0.645,0.1032){\color[rgb]{0,0,0}\makebox(0,0)[lt]{\lineheight{1.25}\smash{\begin{tabular}[t]{l}$=$\end{tabular}}}}%
    \put(0.81,0.1032){\color[rgb]{0,0,0}\makebox(0,0)[lt]{\lineheight{1.25}\smash{\begin{tabular}[t]{l}$+$\end{tabular}}}}%
  \end{picture}%
\endgroup%

		\caption{The hexagon in Example~\ref{ex:esagono_sghembo} and its two maximal Minkowski decompositions}
		\label{fig:esagono_sghembo}
	\end{figure}
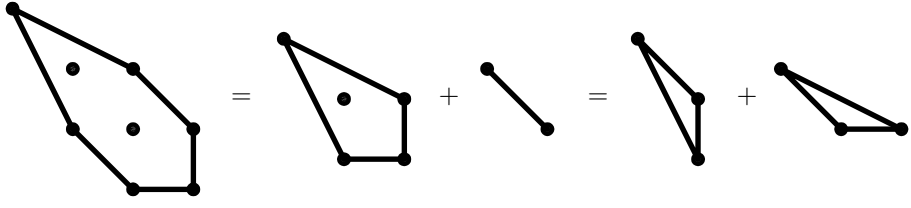
	\[
	\begin{matrix}
		p & q & x & y & z \\
		\hline
		1 & 0 & -1 & -2 & 1 \\
		0 & 1 & 1 & 1 & -2
	\end{matrix}
	\]
	and the algebra $A_F$ is isomorphic to $\CC[x,y,z] / J$, where $J$ is the ideal generated by
	\begin{align*}
		f_2 &= (x+ 2y - z)^2 - x^2 - 2y^2 + z^2 \\
		f_3 &= (x+ 2y - z)^3 - x^3 - 2y^3 + z^3 \\
		f_4 &= (x+ 2y - z)^4 - x^4 - 2y^4 + z^4 \\
		g_2 &= (-x-y+2z)^2 + x^2 + y^2 - 2 z^2 \\
		g_3 &= (-x-y+2z)^3 + x^3 + y^3 - 2 z^3 \\
		g_4 &= (-x-y+2z)^4 + x^4 + y^4 - 2 z^4.
		\end{align*}
Via the linear change of coordinates
\[
\left\{  
\begin{array}{l}
	x = 4u + 4v \\
	y = u + w \\
	z = 3u + 4v + w
\end{array}
\right.
\]
and via a Gr\"obner basis calculation we see that $A_F$ is isomorphic to $\CC[u,v,w]/K$, where $K$ is the ideal generated by  $uv$, $uw+vw$, $u^3$, $v^2w$. Let $R$ denote the completion of $A_F$ with respect to the homogeneous maximal ideal, hence $R$ is isomorphic to
\[
\CC \pow{u,v,w} / (uv, uw+vw, u^3, v^2w).
\]
The primary components of the ideal $K$ are:
\begin{itemize}
	\item $(u+v,v^2)$ with radical $(u,v)$,
	\item $(u,w)$ which is prime,
	\item $(u^3,v,w)$ with radical $(u,v,w)$.
\end{itemize}
The first primary component of $K$ gives the irreducible component of $\Spec A_F$ (or of $\Spec R$ by Lemma~\ref{lem:minimal_primes_completion}(4)) which, under the $1$-to-$1$ correspondence in Theorem~\ref{thm:klaus_versal}(3), corresponds to the Minkowski decomposition of $F$ into the sum of a quadrilateral and a segment (see the middle of Figure~\ref{fig:esagono_sghembo}).
The second primary component of $K$ gives the irreducible component of $\Spec A_F$ which corresponds to the Minkowski decomposition of $F$ into the sum of two triangles (see the right of Figure~\ref{fig:esagono_sghembo}).
The third primary component of $K$ gives an embedded prime of $A_F$.

We want to show that $R$ is not isomorphic to the quotient of a power series ring modulo a monomial ideal.

The degree $2$ part of the ideal $K$, i.e.\ the $\CC$-vector space spanned by $uv$ and $uw+vw$, gives a pencil $\Lambda$ of conics in $\PP^2$ such that the $4$ lines $(u=0)$, $(v=0)$, $(w=0)$, $(u+v =0)$ in $\PP^2$ are the irreducible components of the $2$ reducible members of $\Lambda$. Since $\PP^2$ has only $3$ coordinate lines, it is not possible to have a monomial basis of $\Lambda$ after any linear coordinate change. This implies that the standard graded algebra $A_F$ is not isomorphic to a standard graded monomial algebra.

Let $\frakm$ denote the maximal ideal of $R$. If $R$ were isomorphic to the quotient of a power series ring modulo a monomial ideal, then by Lemma~\ref{lem:minimal_primes_completion}(1) $\mathrm{gr}_\frakm R \simeq A_F$ would be isomorphic to a standard graded monomial algebra; but this was excluded above. Therefore $R$ is not isomorphic to the quotient of a power series ring modulo a monomial ideal.

\end{example}

\subsection{Necessary conditions and characterisation of the deformation spaces of isolated Gorenstein toric $3$-fold singularities}

Let $\compC$ denote the category of noetherian complete local $\CC$-algebras with residue field $\CC$.

\begin{proposition} \label{prop:consequences_thm_altmann}
	Let $R \in \compC$ with maximal ideal $\frakm$ and with embedding dimension $d$.
	If $R$ is the hull of the deformation functor of an isolated Gorenstein toric $3$-fold singularity over $\CC$, then the following statements hold true.
	\begin{enumerate}
		\item If $d \geq 2$, then $\dim_\CC (\frakm^2 / \frakm^3) = (d^2 +d - 4)/2$.
		\item If $R$ is formally smooth over $\CC$, then $R$ is isomorphic either to $\CC$ or to $\CC \pow{x}$.
	\end{enumerate}
\end{proposition}

\begin{proof}
(1) follows from Theorem~\ref{thm:klaus_versal} and Proposition~\ref{prop:hilbert_function_A_F}.
If $R$ is formally smooth over $\CC$, then $\dim_\CC \frakm^2 / \frakm^3 = (d^2 + d)/2$; thus by (1) we deduce $d \leq 1$.
\end{proof}

   \begin{theorem} \label{thm:explicit_versal_Gorenstein_3fold_toric}
	Let $R \in \compC$ with embedding dimension $\leq 2$.
	Then  $R$ is the hull of the deformation functor of an isolated Gorenstein toric $3$-fold singularity if and only if $R$ is isomorphic to (exactly) one of the following $\CC$-algebras:
\begin{enumerate}
	\item[0)] $\CC $,
	\item[1.a)] $\CC [x] / (x^2)$,
	\item[1.b)] $\CC \pow{x}$,
	\item[2.a)] $\CC [x,y]/(x^2, y^2)$,
	\item[2.b)] $\CC [x,y] / (x^2, xy, y^3)$,
	\item[2.c)] $\CC \pow{x,y} / (x^2, xy)$.
\end{enumerate}
\end{theorem}

\begin{proof}
It follows from Theorem~\ref{thm:klaus_versal}, Proposition~\ref{prop:triangles_and_quadrilaterals}, Proposition~\ref{prop:pentagon},
Example~\ref{ex:klaus}, Example~\ref{ex:pentagono_indecomponibile}, Example~\ref{ex:pentagono_strano}.
\end{proof}

\subsection{Deformations of isolated $\QQ$-Gorenstein toric singularities} \label{sec:subsection_toric_singularities_final}

Cyclic quotient surface singularities are exactly toric surface singularities. Their deformations have been studied quite extensively \cite{riemenschneider_cyclic, christophersen_cyclic, stevens_cyclic, behnke_riemenschneide_cyclic, stevens_cyclic_new}.
An easy consequence is the following:

\begin{lemma} \label{lem:surface}
Let $R \in \compC$ with embedding dimension $\leq 3$.
Then $R$ is the hull of the deformation functor of an isolated toric surface singularity if and only if $R$ is isomorphic to one of the following $\CC$-algebras: $\CC \pow{x}$, $\CC \pow{x,y}$, $\CC \pow{x,y,z}$.
\end{lemma}

\begin{proof}
One implication is easy: for each integer $m \geq 1$, the formally smooth algebra $\CC \pow{x_1, \dots, x_m}$ is the hull of the deformation functor of the surface $A_m$-singularity, i.e.\ $\frac{1}{m+1}(1,m)$.

For the other implication, assume that $R$ is the hull of the singularity $X = \frac{1}{n}(1,q)$, where $n$ and $q$ are coprime integers such that $1 \leq q \leq n-1$.
Let $e$ denote the embedding dimension of $X$.
The case $q=n-1$ (i.e.\ $e=3$) gives the $A_{n-1}$-singularity and has been considered above. We can assume $q < n-1$ (i.e.\ $e \geq 4$). One considers
the continuous fraction expansion
\[
\frac{n}{n-q} = [a_2, a_3, \dots, a_{e-1}] = a_2 - \cfrac{1}{a_3 - \cfrac{1}{a_4 - \ddots}}
\]
where $a_2, \dots, a_{e-1}$ are integers $\geq 2$.
 By \cite[Satz~11]{riemenschneider_cyclic} we have
\[
3 \geq \embdim R = \dim \TT^1_X = \sum_{i=2}^{e-1} a_i - 2 = \sum_{i=2}^{e-1} (a_i - 2) + 2e-6 \geq 2e-6.
\]
This implies that we need to have $e = 4$, hence
\[
\frac{n}{n-q} = a_2 - \frac{1}{a_3}
\]
and $3 \geq \embdim R = a_2 + a_3 -2$, i.e.\ $a_2 + a_3 \leq 5$. There are $3$ possibilities.
\begin{itemize}
	\item The continuous fraction expansion $[2,2] =\frac{3}{2}$ is associated to the singularity $X = \frac{1}{3}(1,1)$; this implies that $R$ is isomorphic to $\CC \pow{x,y}$.
\item The continuous fraction expansion $[3,2] = \frac{5}{2}$ is associated to the singularity $X = \frac{1}{5}(1,3)$; this implies that $R$ is isomorphic to $\CC \pow{x,y,z}$.
\item The continuous fraction expansion $[2,3] = \frac{5}{3}$ is associated to the singularity $X = \frac{1}{5}(1,2)$ which is isomorphic to $\frac{1}{5}(1,3)$.
\end{itemize}
This concludes the proof of Lemma~\ref{lem:surface}.
\end{proof}

\begin{remark} \label{rem:klaus_rigidity}
We now mention two consequences of \cite[(6.5)]{altmann_minkowski}:
\begin{enumerate}
\item every isolated $\QQ$-Gorenstein toric singularity of dimension $\geq 4$ is rigid;

\item every isolated $\QQ$-Gorenstein non-Gorenstein toric $3$-fold singularity is rigid.
\end{enumerate}
\end{remark}

\begin{proof}[Proof of Theorem~\ref{thm:main_toric_singularity}]
Combine Theorem~\ref{thm:explicit_versal_Gorenstein_3fold_toric}, Lemma~\ref{lem:surface}, and Remark~\ref{rem:klaus_rigidity}.
\end{proof}

\begin{proof}[Proof of Theorem~\ref{thm:main_no_murphys_law_for_toric_singularities}]
Fix $n \geq 1$ and consider $R = \CC \pow{x_1, \dots, x_n} / (x_1^3)$.
We need to prove that there is no isolated $\QQ$-Gorenstein toric singularity $X$ such that $R$ is the hull of the deformation functor of $X$.
By contradiction let us assume that there is such an $X$.

If $X$ has dimension $2$, then each irreducible component of $\Spf R$ is a smoothing component of $X$; therefore, by openness of versality, we have that every irreducible component of $\Spf R$ is generically smooth, in particular generically reduced. This contradicts the fact that $\Spf \CC \pow{x_1, \dots, x_n} / (x_1^3)$ has a unique irreducible component, which has multiplicity $3$. Therefore $X$ has dimension $\geq 3$.

By Remark~\ref{rem:klaus_rigidity}, $X$ has dimension $3$ and is Gorenstein. Also this situation is impossible by Proposition~\ref{prop:consequences_thm_altmann}(1).
\end{proof}

\section{Deformations of toric Fano varieties}
\label{sec:toric_Fano}
\subsection{From toric singularities to toric Fano varieties} \label{sec:subsection_Fano}

From a lattice polygon we construct a $3$-dimensional lattice polytope:

\begin{construction} \label{constr:P_F}
	If $F$ be a lattice polygon in $\ZZ^2$, then $P_F$ denotes the $3$-dimensional lattice polytope which is the convex hull of $(F \times \{1\}) \cup ((-F) \times \{-1\})$ in $\ZZ^3$.
\end{construction}

The polytope $P_F$ has two facets that are isomorphic to $F$: the top one and the bottom one.

It is clear that the origin lies in the interior of $P_F$ and that every vertex of $P_F$ is a primitive lattice vector of $\ZZ^3$. In other words, $P_F$ is a \emph{Fano polytope}.
Therefore, one can consider the face fan of $P_F$ (which is the fan given by the cones over the faces of $P_F$) and the corresponding toric variety, which is a Fano $3$-fold.

Since $P_F$ is centrally symmetric, the Fano $3$-fold associated to the face fan of $P_F$ is K-polystable by \cite{berman_polystability}.

\begin{proof}[Proof of Theorem~\ref{thm:main_hull_K-moduli_stack}]
Let $Y$ be an isolated Gorenstein toric $3$-fold singularity and let $R$ be the hull of its deformation functor.
The toric affine variety $Y$ is associated to a lattice polygon $F$ in $\ZZ^2$ with unit edges.
Consider the $3$-dimensional lattice polytope $P_F$ as in Construction~\ref{constr:P_F} and the K-polystable toric Fano $3$-fold $X$ associated to the face fan of $P_F$.

Let $Y_+$ (resp.\ $Y_-$) denote the toric affine variety associated to the cone over the top (resp.\ bottom) face of $P_F$.
In this way $Y_+$ and $Y_-$ are isomorphic to $Y$ and  are open neighbourhoods of two isolated points of the singular locus of $X$.

Since $\rH^1(X, \cT^0_X) = 0$ and $\rH^2(X, \cT^0_X) = 0$ by \cite[Proof of Theorem~5.1]{totaro} (see also \cite{petracci_survey}), by Proposition~\ref{prop:section_deformation_qG} we have that the product of restriction maps
\begin{equation} \label{eq:restriction_map_toric}
\DefqG{X} \longrightarrow \DefqG{Y_+} \times \DefqG{Y_-}
\end{equation}
is surjective and admits a section.
As $Y$ is Gorenstein, $ \DefqG{Y_\pm} = \Def{Y_\pm}$.
This proves Theorem~\ref{thm:main_hull_K-moduli_stack} in dimension $n=3$. In higher dimension it is enough to consider the product of $X$ with a projective space, by \cite[Proposition~4.1]{ask_petracci}.
\end{proof}

\begin{remark}
One can prove that away from the singular points of $Y_+$ and of $Y_-$ the stack $\frakX$ is locally unobstructed, i.e.\ $X$ is locally qG-unobstructed.
It is tempting to speculate that the map~\eqref{eq:restriction_map_toric} is smooth.
Unfortunately we were not able to prove this (see Remark~\ref{rmk:restriction_map_smooth?}).
\end{remark}

\begin{example}
Two examples of Construction~\ref{constr:P_F} and of Theorem~\ref{thm:main_hull_K-moduli_stack} are contained in \cite{ask_petracci} and one example is contained in \cite{petracci_agplus}.
In all these 3 examples the map~\eqref{eq:restriction_map_toric} is smooth.
\end{example}

\subsection{Arbitrarily many branches} \label{sec:arbitrarily_many_branches}

Here we prove that the K-moduli stack and the K-moduli space can have arbitrarily many local branches.

\begin{lemma} \label{lem:infinite_iterates}
	Let $e_1$ and $e_2$ be the vectors of the standard basis of the lattice $\ZZ^2$. 
	Consider the matrix
\begin{equation*}
	L = \begin{pmatrix}
		5 & 2 \\ 2 & 1
	\end{pmatrix} \in \mathrm{SL}_2(\ZZ).
\end{equation*}
	Then,  whenever $m$ and $n$ are distinct non-negative integers, the two subsets
	\begin{gather*}
		\left\{ L^m e_1, L^m(e_1+e_2), L^m e_2, - L^m e_1, - L^m(e_1+e_2), - L^m e_2   \right\} \\
		\left\{ L^n e_1, L^n(e_1+e_2), L^n e_2, - L^n e_1, - L^n(e_1+e_2), - L^n e_2   \right\}
		\end{gather*}
				are disjoint.
\end{lemma}

\begin{proof}
Since $L$ maps the first quadrant $(\RR_{\geq 0})^2$ to itself, it is enough to consider the iterates of $e_1$, $e_2$ and $e_1+e_2$ under $L$.

The matrix $L$ is diagonalisable and its eigenvalues are $3+2 \sqrt{2}$ and $3-2 \sqrt{2}$, which are not roots of unity.
Therefore no power of $L$ has an eigenvalue which is equal to $1$.
This implies that the iterates of $e_1$ (resp.\ $e_2$, $e_1+e_2$) are all distinct.

In order to conclude it is enough to show that the three sets $\{L^n e_1 \mid n \in \NN \}$, $\{L^n e_1 \mid n \in \NN \}$, $\{L^n (e_1+e_2) \mid n \in \NN \}$ are pairwise disjoint. This is true because modulo $2$ the matrix $L$ becomes the identity, so the reduction modulo $2$ of every iterate of $e_1$ (resp.\ $e_2$, $e_1+e_2$) is $(1,0)$ (resp.\ $(0,1)$, $(1,1)$) in $(\ZZ / 2 \ZZ)^2$.
	\end{proof}

\begin{proof}[Proof of Theorem~\ref{thm:main_many_branches}]
	Fix an arbitrary integer $r \geq 1$.
	Consider the hexagon $H$ which is the convex hull of $(1,0)$, $(1,1)$, $(0,1)$, $(-1,0)$, $(-1,-1)$, $(0,-1)$ in $\ZZ^2$.
	It is well known that $H$ has two maximal Minkowski decompositions: one expresses $H$ as the sum of two triangles and one expresses $H$ as the sum of three unit segments.
	
	Let $L$ be the matrix considered in Lemma~\ref{lem:infinite_iterates}. Consider the Minkowski sum
	\[
	F = H + LH + L^2 H + \cdots + L^r H.
	\]
	By Lemma~\ref{lem:infinite_iterates} the polygon $F$ has $6r+6$ vertices and has unit edges. Moreover, $F$ has at least $2^{r+1}$ distinct maximal Minkowski decompositions.
	
	Let $A = A_F$ be the standard graded algebra associated to $F$ via Construction~\ref{constr:algebra_A_F}.
	Let $Y$ be the isolated Gorenstein toric $3$-fold singularity associated to the cone over $F$. Let $R$ denote the hull of $\Def{Y}$.
	By Theorem~\ref{thm:klaus_versal} $R$ is the completion of $A$ at its maximal homogeneous ideal and $\Spf R$ has at least $2^{r+1}$ irreducible components.
	By Lemma~\ref{lem:minimal_primes_completion} $A$ has at least $2^{r+1}$ minimal primes.
	
	Consider the $3$-dimensional polytope $P = P_F$ as in Construction~\ref{constr:P_F}.
	Since $H$ and $F$ are centrally symmetric, it is clear that $P$ is a prism over $F$, i.e.\ $P$ is the product of $F$ with a segment of lattice length $2$.
	Let $X$ be the toric Fano $3$-fold associated to the face fan of $P$.
	Let $R'$ denote the hull of $\DefqG{X}$.
	From the surjectivity of \eqref{eq:restriction_map_toric} in the proof of Theorem~\ref{thm:main_hull_K-moduli_stack} we deduce that $R'$ has at least $(2^{r+1})^2 = 2^{2r + 2}$ minimal primes. This shows that the K-moduli stack $\stack{3}$ has at least $2^{2r+2}$ local branches at the point corresponding to $X$.
	
	Now we investigate the local structure of the K-moduli space $\modspace{3}$ near the point corresponding to $X$.
	Let $G$ be the automorphism group of $X$, which is reductive by \cite{ABHLX}.
	By \cite{luna_etale_slice_stacks} the completion of the stalk of the structure sheaf of $\modspace{3}$ at $[X]$ is isomorphic to the invariant subring $(R')^G$ of $R'$ under the action of $G$.

Let $T$ denote the big torus $(\CC^*)^3$ acting on the toric variety $X$.
	Since every facet of the polar of $P$ has no interior lattice points, by \cite[Proposition~2.8]{ask_petracci} $G$ is the semidirect product of the torus $T$ with the automorphism group $\Aut(P)$ of the polytope $P$. In our case we have that $\Aut(P)$ is generated by two reflections: one which swaps top and bottom
	\[
	\begin{pmatrix}
	1 & 0 & 0 \\
	0 & 1 & 0 \\
	0 & 0 & -1
	\end{pmatrix}
	\]
	and one which is the central symmetry on the horizontal plane and keeps the vertical direction fixed
		\[
	\begin{pmatrix}
	-1 & 0 & 0 \\
	0 & -1 & 0 \\
	0 & 0 & 1
	\end{pmatrix}.
	\]
	In particular $\Aut(P)$ is isomorphic to $C_2 \times C_2$, where $C_2$ denotes the cyclic group of order $2$.
	
	The map  \eqref{eq:restriction_map_toric} and its section are clearly $T$-equivariant. Therefore, by taking invariants with respect to the action of $T$, there exist two local ring homomorphisms
	\[
	f \colon (R \widehat{\otimes}_\CC R )^T \longrightarrow (R')^T \quad \text{and} \quad g \colon (R')^T \longrightarrow (R \widehat{\otimes}_\CC R )^T
	\]
	such that $g \circ f = \mathrm{id}_{(R \widehat{\otimes}_\CC R )^T}$.
	Let us better understand the ring $(R \widehat{\otimes}_\CC R )^T$. Let $d$ be the embedding dimension of $R$, so that $R$ is a quotient of the power series ring $\CC \pow{t_1, \dots, t_d}$.
	Therefore $R \widehat{\otimes}_\CC R$ is a quotient of $\CC \pow{t_{1,+}, \dots, t_{d,+}, t_{1,-}, \dots, t_{d,-}}$, where $t_{i,+}$ (resp.\ $t_{i,-}$) denotes a coordinate on $\TT^1_{Y_+}$ (resp.\ $\TT^1_{Y_-}$).
	Recall that $Y_+$ and $Y_-$ are the toric affine open subschemes of $X$ corresponding to the top face and to the bottom face of $P$.
	Moreover $\TT^1_{Y_+}$ and $\TT^1_{Y_-}$ are linear representations of the torus $T$, so they split completely as direct sums of characters of $T$.
	By \cite{altmann_versal} $\TT^1_{Y_+}$ (resp.\ $\TT^1_{Y_-}$) is homogeneous of degree $(-1,0,0) \in \ZZ^3$ (resp. $(1,0,0) \in \ZZ^3$).
	By Remark~\ref{rmk:segre_product_and_torus_actions} $(R \widehat{\otimes}_\CC R)^T$ is the completion of the Segre product $A \# A$.
	By Lemma~\ref{lem:irr_comp_segre_product} $A \# A$ has at least $2^{2r+2}$ minimal primes.
	By Lemma~\ref{lem:minimal_primes_completion} $(R \widehat{\otimes}_\CC R)^T$ has at least $2^{2r+2}$ minimal primes.
	From the existence of $f$ and $g$ we deduce that $(R')^T$ has  at least $2^{2r+2}$ minimal primes.
	Since $T$ has index $4$ in $G$, we deduce that $(R')^G$ has at least $2^{2r+2}/4 = 2^{2r}$ minimal primes.
	
	In conclusion, $\stack{3}$ has at least $2^{2r+2}$ local branches at $[X]$ and $\modspace{3}$ has at least $2^{2r}$ local branches at $[X]$.
	This proves Theorem~\ref{thm:main_many_branches} in dimension $n=3$. In higher dimension it is enough to consider the product of $X$ with a projective space.
	\end{proof}

\bibliography{Biblio_murphy}


\end{document}